\documentclass[EJP,preprint]{ejpecp} 

\usepackage{etoolbox,comment,bbm,mathrsfs}
\usepackage{makeidx}
\usepackage{enumerate}
\usepackage[normalem]{ulem}

\def\dpd{\Delta} 

\def\var{\mathop{\rm Var}}

\newcommand{\qe}{E_\omega}

\newcommand{\dd}{\mathrm{d}}

\newcommand{\B}{\mathbb{B}}
\newcommand{\K}{\mathbb{K}}
\newcommand{\R} {\mathbb{R}}
\newcommand{\Q} {\mathbb{Q}}
\newcommand{\Z} {\mathbb{Z}}
\newcommand{\nn}{\nonumber}
\newcommand{\N} {\mathbb{N}}
\newcommand{\dist} {\textnormal{dist}}
\newcommand{\tr}{{\rm tr}}
\newcommand{\diag}{{\rm diag}}

\DeclarePairedDelimiter{\norm}{\lVert}{\rVert}
\providecommand{\Abs}[1]{\Bigr\lvert#1\Bigl\rvert}
\providecommand{\floor}[1]{\lfloor#1\rfloor}
\providecommand{\ceil}[1]{\lceil#1\rceil}
\providecommand{\mb}[1]{\mathbb{#1}}
\providecommand{\mc}[1]{\mathcal{#1}}
\providecommand{\ms}[1]{\mathscr{#1}}

\SHORTTITLE{Quantitative homogenization in a balanced RE}

 \TITLE{Quantitative homogenization in a balanced random environment
\support{XG is supported by Simons Foundation's Collaboration Grant for Mathematicians \# 852943. JP was partially supported by NSA grants H98230-15-1-0049 and H98230-16-1-0318. HT is supported in part by NSF grant DMS-1664424.}
} 

\AUTHORS{%
 Xiaoqin Guo\footnote{Department of Mathematical Sciences, University of Cincinnati, \EMAIL{guoxq@ucmail.uc.edu}}\and 
 Jonathon Peterson  \footnote{Department of Mathematics, Purdue University, \EMAIL{peterson@purdue.edu}}
\and
 Hung V. Tran \footnote{Department of Mathematics, University of Wisconsin Madison,\EMAIL{hung@math.wisc.edu}}
 }


\AMSSUBJ{35J15 
35J25 
35K10 
35K20 
60G50 
60K37 
74Q20 
76M50. 
} 

\VOLUME{0}
\YEAR{2020}
\PAPERNUM{0}
\DOI{10.1214/YY-TN}

\ABSTRACT{We consider discrete non-divergence form difference operators in a
 random environment and the corresponding process -- the random walk in a balanced random environment in $\Z^d$ with a finite range of dependence. We first quantify the ergodicity of the {\it environment  from the point of view of the particle}. As a consequence, we quantify the quenched central limit theorem of the random walk with an algebraic rate. Furthermore, we prove an algebraic rate of convergence for the homogenization of the Dirichlet problems for both elliptic and parabolic non-divergence form difference operators. }

\begin{document}




\section{Introduction}\label{sec:intro}
Let $\mb S_{d\times d}$ denote the set of $d\times d$ positive-definite diagonal matrices. A map 
\[
\omega:\Z^d\to\mb S_{d\times d}
\] is called an {\it environment}. Denote the set of all environments by $\Omega$ and let $\mb P$ be a probability measure on $\Omega$.
Expectation with respect to $\mb P$ is denoted by $\mb E$. 

Let $\{e_1,\ldots,e_d\}$ be the  canonical basis for $\R^d$.  
For any function $u:\Z^d\to\R$ and 
\[
\omega=\left\{\omega(x)=\mathrm{diag}[\omega_1(x),\ldots, \omega_d(x)], x\in\Z^d\right\}\in\Omega,
\]
define the non-divergence form difference operator
\begin{align*}
\tr(\omega(x)\nabla^2 u)=\sum_{i=1}^d\omega_i(x)[u(x+e_i)+u(x-e_i)-2u(x)],
\end{align*}
where $\nabla^2=\diag[\nabla^2_1,\ldots, \nabla_d^2]$, and 
$\nabla_i^2 u(x)=u(x+e_i)+u(x-e_i)-2u(x)$. 

For $r>0$,  with $|x|:=|x|_2$,  we let
\[
\B_r=\left\{x\in\R^d: |x|<r\right\}, \quad
B_r=\B_r\cap\Z^d
\]
denote the continuous and discrete balls with center $o=(0,\ldots,0)$ and radius $r$, respectively. For any $B\subset\Z^d$,  its {\it discrete boundary} is the set
\[
\partial B:=\left\{z\in\Z^d\setminus B: \dist(z,x)=1 \text{ for some }x\in B\right\},
\]
where $\dist(z,x):=|z-x|_1$. 
Let $\bar B=B\cup\partial B$. Note that with abuse of notation, whenever confusion does not occur,  we also use $\partial A$ and $\bar A$ to denote the usual continuous boundary and closure of $A\subset\R^d$, respectively.

 For $x\in\Z^d$, a {\it spatial shift} $\theta_x:\Omega\to\Omega$ is defined by 
 \[
 (\theta_x\omega)(\cdot)=\omega(x+\cdot).
 \]
In a random environment $\omega\in\Omega$, we consider the discrete elliptic Dirichlet problem
\begin{equation}\label{eq:elliptic-dirich}
\left\{
\begin{array}{lr}
\tfrac 12\tr(\omega\nabla^2u(x))=\frac{1}{R^2}f\left(\tfrac{x}{R}\right)\psi(\theta_x\omega) & x\in B_R,\\[5 pt]
u(x)=g\left(\tfrac{x}{|x|}\right) & x\in \partial B_R,
\end{array}
\right.
\end{equation}
where $f\in\R^{\B_1}, g\in\R^{\partial\B_1}$ are functions with good regularity properties and $\psi\in\R^\Omega$ is bounded and satisfies suitable measurability condition. Stochastic homogenization studies  (for $\mb P$-almost all $\omega$)  the convergence of $u$ to the solution $\bar u$ of a deterministic {\it effective equation}
\begin{equation}\label{eq:effective-ellip}
\left\{
\begin{array}{lr}
\tfrac 12\tr(\bar a D^2\bar{u})
=f\bar\psi &\text{ in }\B_1,\\
\bar u=g &\text{ on }\partial \B_1,
\end{array}
\right.
\end{equation}
as $R\to\infty$. Here $D^2 \bar{u}$ denotes the Hessian matrix of $\bar{u}$ and $\bar a=\bar a(\mb P)\in\mb S_{d\times d}$ and $\bar\psi=\bar\psi(\mb P,\psi)\in\R$ are {\it deterministic} and do not depend on the realization of the random environment (see the statement of Theorem \ref{thm:homog-nondiv} for formulas for $\bar{a}$ and $\bar{\psi}$).

Similarly we can also formulate the parabolic version of the discrete Dirichlet problem. To this end, we need some notations. Denote {\it parabolic cylinders} by
\[
\K_R:=\B_R\times[0,R^2)\subset\R^d\times\R, \quad 
K_R=\K_R\cap(\Z^d\times\Z),
\]
and their {\it parabolic boundaries} as
\begin{align*}
&\partial^p\K_R=\left(\partial\B_R\times(0,R^2)\right)\cup\left(\bar{\B}_R\times\{R^2\}\right), \\
&\partial^pK_R=\left(\partial B_R\times\{1,\ldots,\ceil{R^2}\}\right)\cup \left(B_R\times \{\ceil{R^2}\}\right)=:\partial^l K_R\cup\partial^t K_R.
\end{align*}
Here $\partial^l, \partial^t$ denote lateral- and time- boundaries. Write
\[
\bar\K_R=\K_R\cup\partial^p\K_R, \quad \bar K_R=K_R\cup\partial^p K_R.
\]
We will also consider the homogenization of the discrete parabolic problem
\footnote{Note that here the discrete Hessian $\nabla^2$ only acts on the space coordinate $x$.}
\begin{equation}\label{eq:para-dirich}
\begin{cases}
 \tfrac 12\tr(\omega\nabla^2u(x,n+1))+[u(x,n+1)-u(x,n)]=\tfrac{1}{R^2}f\left(\tfrac xR,\tfrac{n}{R^2}\right)\psi(\theta_x\omega) \quad\text{ in } K_R,\\
u(x,n)=g\left(\tfrac{x}{|x|\vee\sqrt n}, \tfrac{n}{|x|^2\vee n}\right) \qquad\qquad\text{on }\partial^p K_R, 
\end{cases}
\end{equation}
as $R\to\infty$ to an effective equation
\begin{equation}
\label{eq:effective-para}
\left\{
\begin{array}{lr}
\tfrac12\tr(\bar a D^2\bar u)+\bar b\partial_t\bar u=f\bar\psi\quad&\text{ in }\K_1\\
\bar u=g \quad&\text{ on }\partial^p\K_1,
\end{array}
\right.
\end{equation}
where $f\in\R^{\K_1}, g\in\R^{\partial^p\K_1}$, $\psi\in\R^\Omega$ are functions with suitable regularity and measurability, and $\bar a, \bar b, \bar\psi$ are deterministic.

The difference equations \eqref{eq:elliptic-dirich} and \eqref{eq:para-dirich} are used to describe random walks in a random environment (RWRE) in $\Z^d$. To be specific, 
we set
\begin{equation}\label{def:omegaxei}
\omega(x,x\pm e_i):=\frac{\omega_i(x)}{2\tr\omega(x)} \quad  \text{ for } i=1,\ldots d,
\end{equation}
 and $\omega(x,y)=0$ if $|x-y|\neq 1$. Namely, we normalize $\omega$ to get a transition probability.  We remark that the configuration of $\{\omega(x,y):x,y\in\Z^d\}$ is also called a {\it balanced environment} in the literature \cite{L-82,GZ-12,BD-14,DGR-15}.
For a fixed $\omega\in\Omega$, the random walk  $(X_n)_{n\ge 0}$ in the environment $\omega$ is a Markov chain  in $\Z^d$ with transition probability $P_\omega^x$ specified by
\begin{equation}\label{eq:def-RW}
P_\omega^x\left(X_{n+1}=z|X_n=y\right)=\omega(y,z).
\end{equation}
The expectation with respect to $P_\omega^x$ is written as $E_\omega^x$. When the starting point of the random walk is $0$, we sometimes omit the superscript and simply write $P_\omega^0, E_\omega^0$ as $P_\omega$ and $E_\omega$, respectively.  Notice that for random walks $(X_n)$ in an environment $\omega$, 
\begin{equation}\label{def:omegabar}
\bar\omega^i=\theta_{X_i}\omega\in\Omega, \quad i\ge 0,
\end{equation}
is also a Markov chain, called the {\it environment from the point of view of the particle}. With abuse of notation, we enlarge our probability space so that $P_\omega$ still denotes the joint law of the random walks and $(\bar\omega^i)_{i\ge 0}$. 

The following quenched central limit theorem (QCLT) is proved by Lawler \cite{L-82}, which is a discrete version of Papanicolaou,  Varadhan \cite{PV-82}.
\begin{theorem}[Lawler \cite{L-82}]\label{thm:QCLT}
Assume that the law $\mb P$ of the environment is ergodic under spatial shifts $\{\theta_x:x\in\Z^d\}$ and that $\mb P(\omega(0,\pm e_i)\ge \kappa, i=1,\ldots,d)=1$ for some constant $\kappa>0$. Then 
\begin{enumerate}[(i)]
\item There exists a probability measure $\mb Q$ that is mutually absolutely continuous with respect to $\mb P$ such that $(\bar\omega^i)_{i\ge 0}$ is an ergodic (with respect to time shifts) sequence under law $\mb Q\times P_\omega$.
\item For $\mb P$-almost every $\omega$, the rescaled path $X_{n^2t}/n$ converges weakly (under law $P_\omega$) to a Brownian motion with covariance matrix $\bar a=E_\Q[\omega(0)/\tr\omega(0)]>0$. 
\end{enumerate}
\end{theorem}

This QCLT is later generalized to (non-uniformly) elliptic ergodic environment with a moment condition by Guo, Zeitouni \cite{GZ-12}, and genuinely $d$-dimensional i.i.d.\,environment without ellipticity by Berger, Deuschel \cite{BD-14}. For time-dependent balanced environments, the QCLT is proved by Deuschel, Guo, and Ramirez \cite{DGR-15}.

We remark that the QCLT is obtained for very few RWRE models with zero effective speed. 
Another case that QCLT is proved for zero-speed RWRE is the random conductance model, cf.\ the survey article by Biskup \cite{MB-11} and references therein. 
Note that unlike the QCLT of random conductance models, for balanced RWRE the invariant measure  $\Q$ of the environment as viewed from the particle does not have an explicit formula in terms of the environment measure $\mb P$. Even though by Birkhoff's ergodic theorem, $\Q$ can be approximated qualitatively by
\begin{equation}\label{eq:birkhoff}
\lim_{n\to\infty} \frac{1}{n} \sum_{i=0}^{n-1}E_\omega[\psi(\bar\omega^i)] = E_\Q[\psi] \quad \mb P\text{-a.s.}
\end{equation}
for any bounded function $\psi$ on environments, in order to better understand the effective matrix $\bar a$ it is important to quantify the speed of this convergence.

The difference equations \eqref{eq:elliptic-dirich}, \eqref{eq:para-dirich} and PDEs \eqref{eq:effective-ellip}, \eqref{eq:effective-para} are used to describe microscopic and macroscopic dynamics of a diffusive particle, respectively. For instance,  the solution of the Dirichlet problem \eqref{eq:elliptic-dirich} can be represented in terms of the RWRE:
\begin{equation}\label{eq:FK}
u(x)=E_\omega^x[g(X_\tau/|X_\tau|)]-\frac{1}{R^2}E_\omega^x\big[\sum_{i=0}^{\tau-1}f(\tfrac{X_i}{R})\psi(\bar\omega^i)\big],
\end{equation}
where $\tau=\tau_R=\min\{n\ge 0:X_n\notin B_R\}$. On the other hand, it is well-known that by the classical Feynman-Kac formula, the solution of the PDE \eqref{eq:effective-ellip} can be expressed similarly in terms of the Brownian motion with covariance matrix $\bar a$. The goal of this paper is to exploit this connection to quantify the rate of the micro-to-macro convergence  for both the equations and the processes.

\medskip

\noindent {\bf Throughout the paper, we assume}
\begin{enumerate}
\item [(A1)]
The measure $\mb P$ is translation-invariant under  shifts $\{\theta_x:x\in\Z^d\}$, and $\mb P$ has a finite range $\dpd>0$ of dependence.  That is,  for any subsets $A, B\subset\Z^d$ with $\dist(A,B)=\inf\{\dist(x,y): x\in A, y\in B\}\ge\dpd$,  the collections of variables $\{\omega(x):x\in A\}$ and $\{\omega(y):y\in B\}$ are independent.
\item[(A2)] $\frac{\omega}{\tr\omega}\ge 2\kappa{\rm I}$ for $\mb P$-almost every $\omega$ and some constant 
$\kappa>0$.
\end{enumerate} 
In the paper, we use $c, C$ to denote positive constants which may change from line to line but that only depend on the dimension $d$, the ellipticity constant $\kappa$,  and the range $\dpd$ of dependence unless otherwise stated.

\subsection{Main results}

Our first main result quantifies the speed of convergence in the ergodic averaging \eqref{eq:birkhoff}. 
 Recall  $\dpd$ in (A1). 
We say that $\psi\in\R^\Omega$ is a {\it local function} if it satisfies
\begin{enumerate}
\item[(A3)] 
      \begin{itemize}
      \item $\psi$ is measurable;
      \item $\omega(x)=\tilde{\omega}(x)$ for all $x\notin B_\dpd$ implies $\psi(\omega)=\psi(\tilde\omega)$.
      \end{itemize}
\end{enumerate}
Sometimes we may replace (A3) with the following assumption.
\begin{enumerate}
\item[(A4)] $\psi$ satisfies (A3), $\norm{\psi}_\infty=1$, and $E_{\Q}[\psi]=0$.
\end{enumerate}
\begin{theorem}\label{thm:quan-ergo}
Assume (A1), (A2), (A3).
Recall the notation $\bar\omega^i$ in \eqref{def:omegabar}. 
 For any $p\in(0,d)$, there exist positive constants $\alpha, C, c$ depending only on $(d,\kappa,p,\dpd)$ such that for 
any stopping time $T$ of the random walk $(X_i)_{i\ge 0}$,
 \[
 \mb P\left(\Abs{\frac 1n E_\omega\big[\sum_{i=0}^{T\wedge n-1}(\psi(\bar\omega^i)-E_\Q\psi)\big]}\ge C\norm{\psi}_\infty n^{-\alpha}\right)\le Ce^{-cn^{p/2}}.
 \]
 In particular,
\[
\mb P\left(\Abs{\frac 1n\sum_{i=0}^{n-1}E_\omega[\psi(\bar\omega^i)]-E_\Q\psi}\ge C\norm{\psi}_\infty n^{-\alpha}\right)\le Ce^{-cn^{p/2}}.
\]
\end{theorem}


As a consequence of Theorem~\ref{thm:quan-ergo}, we obtain a Berry-Esseen type estimate for the one-dimensional projections of the QCLT (Theorem~\ref{thm:QCLT}).
\begin{theorem}
\label{thm:BE}
Assume (A1), (A2).
For any $p\in(0,d)$, there exists a constant $\gamma=\gamma(p, d,\kappa,\dpd)>0$ such that for any unit vector $\ell\in\R^d$,
with $\mb P$-probability at least $1-Ce^{-n^{p/2}}$, 
\[
\sup_{r\in\R}
\Abs{
P_\omega
\left(
X_{n}\cdot\ell/\sqrt n \le r\sqrt{\ell^T\bar a\ell}
\right)
-\Phi(r)
}
\le 
Cn^{-\gamma},
\]
where $\Phi(r)=(2\pi)^{-1/2}\int_{-\infty}^r e^{-x^2/2}\dd x$ for all $r \in \R$.
\end{theorem}

\begin{remark}
This is a quantification of the QCLT in Theorem~\ref{thm:QCLT}.
 Previously, for reversible RWRE models, quantitative CLTs were proved by Mourrat \cite{JM-12} and Andres,  Neukamm \cite{AN-17} in the case of the random conductance model, and by Ahn, Peterson \cite{AP-17} in the case of one-dimensional i.i.d.\,environments. 
 For non-reversible RWRE, quantitative CLTs were obtained by Guo,  Peterson \cite{GP-17} for certain ballistic RWRE.

 Quantitative CLTs \cite{AP-17,GP-17} for ballistic RWRE were obtained by
 comparing the random path to sum of independent random variables. 
 However, in the zero-speed regime, due to the complicated correlation between the path and the environment,  there is no such independence structure to exploit.
 Here the quantitative control (Theorem~\ref{thm:quan-ergo}) on the ergodicity of the environment as viewed from the particle  plays a key role.
\end{remark}

Finally, our last two main results give algebraic convergence rates for the stochastic homogenization of the discrete elliptic and parabolic difference equations in \eqref{eq:elliptic-dirich} and \eqref{eq:para-dirich}, respectively.

\begin{theorem}\label{thm:homog-nondiv}
 Assume (A1), (A2), (A3).
Recall the measure $\mb Q$ in Theorem~\ref{thm:QCLT}.
Suppose\footnote{Readers may refer to \cite[Chapter 4.1]{GiTr} for definitions of $C^{0,1}$ and the associated norms.} $g\in C^{3}(\partial \B_1), f\in C^{0,1}(\B_1)$, and $\psi$ 
satisfies $\norm{\psi/\tr\omega(0)}_{\infty}<\infty$. 
For any $q\in(0,d)$, there exist a random variable $\ms X=\ms X(\omega, q,d,\kappa)$
with $\mb E[\exp(c\ms X^d)]<\infty$, and positive
constants $\beta=\beta(d,\kappa,q, \dpd)$ and $C=C(d,\kappa, \dpd,\norm{f}_{C^{0,1}(\B_1)}+\norm{g}_{C^{3}(\partial \B_1)}+\norm{\psi/\tr\omega}_\infty)$ such that
for all $R>0$,  the solution $u$ of \eqref{eq:elliptic-dirich} satisfies 
\[
\max_{x\in B_R}\left|u(x)-\bar u\left(\tfrac{x}{R}\right)\right|
\le 
C(1+\ms X R^{-q/d})R^{-\beta},
\]
where $\bar u$ is the solution of the effective equation \eqref{eq:effective-ellip} with
$\bar a=E_{\mb Q}[\omega/\tr\omega]>0$ and $\bar\psi=E_{\Q}[\psi/\tr\omega]$. 
In particular,
\[
\mb P\left(\max_{x\in B_R}\left|u(x)-\bar u\left(\tfrac{x}{R}\right)\right|\ge 2C R^{-\beta}\right)
\le C\exp(-c R^q).
\]
\end{theorem}

\begin{theorem}
\label{thm:homog-parab}
Assume (A1), (A3), and that
 $\kappa I\le \omega(0)\le \kappa^{-1}I$ for $\mb P$-almost all $\omega$.
Suppose\footnote{Here, $C_2^3$ means ``$C^3$ in space and $C^2$ in time". } $f\in C^{0,1}(\K_1)$, $g\in C_2^3(\partial^p\K_1)$.
For any $q\in(0,d)$, there exist positive constants $\beta=\beta(d,\kappa,\dpd, q)$, $C=C(d,\kappa,\dpd, \|f\|_{C^{0,1}(\K_1)} + \|g\|_{C^3_2(\partial^p\K_1)}+\norm{\psi}_\infty)$  and a random variable $\ms Y=\ms Y(\omega, q,d,\kappa)$ with $\mb E[\exp(c\ms Y^{d+1})]<\infty$ such that for all $R>0$, the solution $u$ of \eqref{eq:para-dirich} satisfies
\[
\max_{K_R}\left|u(x,n)-\bar u\left(\tfrac xR,\tfrac n{R^2}\right)\right|\le C(1+\ms Y R^{-q/(d+1)})R^{-\beta},
\]
where $\bar u$ is the solution of the effective equation \eqref{eq:effective-para} with $\bar a=E_\Q[\omega/\tr\omega]>0$, $\bar b=E_\Q[1/\tr\omega]$ and $\bar\psi=E_\Q[\psi/\tr\omega]$. 
In particular,
\[
\mb P\left(\max_{K_R}\left|u(x,n)-\bar u\left(\tfrac xR,\tfrac n{R^2}\right)\right|\ge 2CR^{-\beta}\right)
\le 
C\exp(-c R^q).
\]
\end{theorem}

In the PDE setting, qualitative results for the homogenization of linear non-divergence form operators were first obtained by Papanicolaou, Varadhan \cite{PV-82}, and Yurinskii \cite{Yur-82}. 
Qualitative results in fully nonlinear setting was obtained by Caffarelli, Souganidis,  and Wang \cite{CSW-05}. 
In terms of quantitative results, Yurinski derived a second moment estimate of the homogenization error in \cite{Yur-88} for linear elliptic case, and Caffarelli, Souganidis \cite{CaSu10} proved a logarithmic convergence rate for the nonlinear elliptic case.  
Afterwards,  Armstrong, Smart \cite{AS-14}, and Lin, Smart \cite{LS-15} achieved an algebraic convergence rate for fully nonlinear elliptic equations, and fully nonlinear parabolic equations, respectively. 
Armstrong, Lin \cite{AL-17} obtained quantitative estimates for the approximate corrector problems. 
Note that apart from our parabolic result (Theorem~\ref{thm:homog-parab}) being discrete, there are two main differences between our case  and the case considered in \cite{LS-15}. The first is that our environment is not time-dependent. The second is that our environment measure does not decorrelate in time as assumed in \cite{LS-15}. 
We remark that there are other quantitative stochastic homogenization results in non-reversible RWRE settings which are different from ours, see e.g.  \cite{BrKu, SzZe, BoZe,BaBo, Baur}, to name a few.

Our work 
is inspired by Armstrong, Smart \cite{AS-14}, and Berger, Cohen, Deuschel, and Guo \cite{BCDG-18}. Specifically,  our Theorem \ref{thm:quan-ergo} can be viewed as a discrete version of the result of Armstrong, Smart \cite{AS-14} in the PDE setting, and our proof depends heavily on the idea  of  \cite{AS-14} which obtained the algebraic rate by investigating sub-additive structure of the convexity of solutions.

Before proceeding with the proofs of the main results, we give here an  outline of the structure of the rest of the paper.

In Section \ref{sec:quan-ergo} we 
quantify the ergodicity of the environment from the point of view of the particle by
proving a quantitative homogenization result (Proposition~\ref{prop:conv_rate_evfpvp}) for a special case of the elliptic problem \eqref{eq:elliptic-dirich} when $f \equiv 1$, $g\equiv 0$.
To this end, we control the homogenization error with 
a subadditive quantity $\mu_n(0)$  introduced by Armstrong, Smart \cite{AS-14}
that measures the convexity of super-solutions in boxes with sidelength $3^n$. 
The main task is to obtain exponential decay for moments of $\mu_n(0)$. A key observation is that it suffices to have a lower bound and appropriate upper bounds in ``nearly homogenized" scales for small perturbations $\mu_n(s)$ of $\mu_n(0)$ for some $s>0$. 
  
To get the lower bound, using an idea of Berger \cite{B-per}, 
we show that the homogenization error  grows subquadratically
by using the ergodicity of the environment viewed from the particle $(\bar\omega_n)$ (see \eqref{def:omegabar}).
In fact, this argument, which is robust, will be later used to quantify the convergence rates in Theorems~\ref{thm:homog-nondiv} and \ref{thm:homog-parab}.

The upper bound is achieved by comparing super-solutions that are ``convex at most points" to a paraboloid. This idea appeared in \cite{Ca-90, M-15} for Monge-Amp\`ere equations, and then was used in \cite{AS-14} for the homogenization of non-divergence form PDEs.  
By exploiting the geometry of the subdifferential set, we obtain what we believe a geometrically clearer (and much shorter) proof compared to the proof of a similar bound in \cite{AS-14}. (See Theorem~\ref{thm:190225}, which does the job of Lemmas 3.1, 3.2, 3.3, Corollary 3.4 and part of the proof of Lemma 4.1 in \cite{AS-14}.) Furthermore, with the upper bound in ``nearly homogenized" scales, we deduce the upper bound (Theorem~\ref{thm:conv_rate_subdiff}) of the subdifferential and avoid complicated induction arguments as in the proof of \cite[Theorem 2.9]{AS-14}.
Note that apart from these technical differences,  our strategy follows closely that of  \cite{AS-14}.

Section~\ref{sec:quan-RWRE} is devoted to the proof of the quantitative QCLT results  Theorems \ref{thm:quan-ergo} and \ref{thm:BE} for the RWRE.
In Section \ref{sec:quan-ellip-para}, 
using quantitative versions of Berger's  argument \cite{B-per} (cf. Lemma~\ref{lem:conv_rate_evfpvp}), we obtain algebraic rates of homogenization for both elliptic and parabolic difference operators (Theorems \ref{thm:homog-nondiv} and \ref{thm:homog-parab}) via quantifying precisely how long it takes the RWRE to behave like a Brownian motion.
Roughly speaking, we will use a random walk up to time $n\ll R^2$ to explore the ``flatness" of the error $h_R(x)=u(x)-\bar u(\tfrac xR)$.
Taking the simpler case $\psi\equiv 1$ in the elliptic problem \eqref{eq:elliptic-dirich} for example,  $E_\omega[u(X_n)-u(0)]\stackrel{\eqref{eq:FK}}{=}\tfrac{1}{R^2}E_\omega[\sum_{k=0}^{n-1}f(\tfrac{X_k}R)]\approx \tfrac{n}{2R^2}\tr\bar a D^2\bar u(0)$. Moreover, by Taylor expansion, 
\[
E_\omega[\bar u(\tfrac{X_n}{R})-\bar u(0)]\approx E_\omega[\tfrac 1{R}X_n\cdot D\bar u(0)+\tfrac1{2R^2}X_n^T D^2\bar u(0)X_n]=\tfrac1{2R^2}E_\omega[\tr(X_n X_n^TD^2\bar u(0))],
\] 
where  $X_n^T$ is the transpose of the column vector $X_n$. Thus, $E_\omega[h_R(X_n)-h_R(0)]\approx\tfrac{n}{2R^2}\tr [(M_n-\bar a)D^2\bar u(0)]$, where $M^{(n)}=E_\omega[X_nX_n^T]/n$ is the covariance matrix of the rescaled random walk $X_n/\sqrt n$. In other words, to measure how flat the homogenization error is, it suffices to control $(M^{(n)}-\bar a)$ which is the difference between the diffusivity of a large scale RWRE and the diffusion matrix $\bar a$ of the limiting Brownian motion.

\section{Quantification of ergodicity of the environment viewed from the particle
}\label{sec:quan-ergo}

For $\omega \in \Omega$, define the operator $L_\omega$ by
\begin{equation}\label{eq:def-of-L}
L_\omega u(x)=\sum_y\omega(x,y)[u(y)-u(x)]=\frac{1}{2\tr\omega(x)}\tr(\omega(x)\nabla^2 u).
\end{equation}
Fix a bounded   local function $\psi$. With a slight abuse of notation we also use $\psi$ to denote the function on $\Z^d$ defined by 
\begin{equation}\label{def:psi}
\psi(x)=\psi_\omega(x)=:\psi(\theta_x\omega).
\end{equation}
For any finite subset $B\subset\Z^d$, consider the Dirichlet problem
\begin{equation}\label{eq:corrector}
\left\{
\begin{array}{lr}
L_\omega\phi= \psi_\omega-E_{\Q}\psi \quad &\text{ in } B,\\
\phi|_{\partial B}=0.
\end{array}
\right.
\end{equation}

The purpose of this section is to obtain the following proposition which states that $\phi$ grows subquadratically in terms of the diameter of $B$.
\begin{proposition}\label{prop:conv_rate_evfpvp}
Assume (A1), (A2), (A3).
For any $p\in(0,d)$, there exists $\alpha=\alpha(d,\kappa,p, \dpd)>0$ such that 
for any $B\subset \square_R =\{x\in\Z^d: |x|_\infty< R/2\}$, the solution $\phi$ of 
\eqref{eq:corrector} 
satisfies
\[
\mathbb P\left(\max_{B}\frac{1}{R^2}|\phi|\ge C\norm{\psi}_\infty R^{-\alpha}\right)
\le  
C\exp(-cR^p).
\]
\end{proposition}

Notice that if we let $\tau=\tau_B=\inf\{n\ge 0: X_n\notin B\}$ be the exit time from $B$. Then the solution $\phi=\phi_{B,\psi}$ of \eqref{eq:corrector} can be expressed as
\[
\phi(x)=- E_\omega^x \left[\sum_{i=0}^{\tau-1}(\psi(\bar\omega^i)-E_{\Q}\psi)\right], \quad x\in \bar B.
\]
Moreover, since $|X_n|^2-n$ is a martingale we have that the expected exit time
\begin{equation}\label{eq:exittime}
 E_\omega^x[ \tau ] = E_\omega^x\left[|X_{\tau}|^2\right] - |x|^2
\end{equation}
is at most $C R^2$ if $B \subset \square_R$.

\subsection{Measuring the convexity of solutions}

To obtain bounds for $\phi$, we use a discrete version of the classical Alexandrov-Bakelman-Pucci (ABP) estimate to control functions with their {\it subdifferentials}.
In this subsection, we will define the subdifferential set and discuss some of its basic properties that will be used in the rest of the paper.

\begin{definition}
For $B\subset\Z^d$, we define for $x\in B$, $u\in\R^{\bar B}$,  the {\it sub-differential} set 
\[
\partial u(x;B)=\left\{p\in\R^d: u(x)-p\cdot x\le u(y)-p\cdot y \text{ for all }y\in\bar B\right\}\subset\R^d.
\]
For any $A\subset B$, let  
\[
\partial u(A;B)=\bigcup_{x\in A}\partial u(x;B).
\] 
We  write $\partial u(B;B)$  simply as $\partial u(B)$.
\end{definition}

\begin{lemma}[ABP inequality]\label{lem:abp}
Let $E\subset\Z^d$ be a finite connected subset, and let ${\rm diam}(\bar E)=\max\{|x-y|:x,y\in \bar E\}$ be the diameter of $E$. There exists a constant $C=C(d)$ such that for any function $u$ on $\bar E$, we have
\[
\min_{\partial E}u\le\min_E u+C{\rm diam}(\bar E)|\partial u(E)|^{1/d}.
\]
Here, for  $U\subset\R^d$,  $|U|$ denotes the Lebesgue measure of $U$.
\end{lemma}
\begin{proof}
Without loss of generality, assume that 
\[M:=\min_{\partial E}u-\min_E u>0,\]
and $u(x_0)=\min_E u$ for some $x_0\in E$. Then, for any $p\in\R^d$ such that
\[
|p| < M/{\rm diam}(\bar E),
\]
we have
\[
u(y)-u(x_0)\ge M > p\cdot(y-x_0)\quad\text{ for all }y\in\partial E.
\]
Thus the minimum of $u(x)-p\cdot x$ is achieved in $E$, and hence, $p\in \partial u(E)$. Therefore, 
\[
\B_{M/{\rm diam}(\bar E)}\subset \partial u(E),
\]
and the lemma follows.
\end{proof}

In our setting, the volume of the subdifferential set is used to measure the convexity of the function. 
For $u\in\R^{\bar B}$, 
let $\breve{u}:\R^d\to\R$ (called the {\it convex envelope} of $u$) denote
the biggest convex function that is smaller than $u$.  That is, 
\begin{equation}\label{eq:def-envelope}
\breve u(x)=\breve u(x;B)=\sup\left\{\ell(x): \ell\text{ is affine and $\ell\le u$ in }\bar B\right\}.
\end{equation}
 Notice that the convex envelope $\breve u$ is defined over the whole $\R^d$.
Here are some basic facts about the subdifferential. See the book of Caffarelli, Cabr\'e \cite{CaCa} for more details in the continuous setting.
\begin{enumerate}
\item The volume $|\partial u(x; B)|$ of a subdifferential set is preserved by  affine translations. That is, letting $\tilde u(x):=u(x)+a\cdot x+b$, then $\partial \tilde{u}(x, B)=\partial u(x; B)+a$ is only a translation of $\partial u(x; B)$, and therefore, the volume is preserved. 
\item If $u(x)=\breve u(x)$, then $\partial\breve u(x;B)=\partial u(x; B)$.   If $u(x)\neq \breve u(x)$, then $|\partial \breve{u}(x; B)|=0$.  Hence, $|\partial u(A;B)|=|\partial{\breve u}(A;B)|$ for any $A\subset B$.
\item The intersection of subdifferentials at different points has Lebesgue measure 0. That is,  $|\partial u(x;B)\cap\partial u(y;B)|=0$ if $x\neq y$. So for $A\subset B$,
\[
\Abs{\partial u(A;B)}
=\sum_{x\in A}|\partial u(x;B)|.\]
\item For any convex function $w:\R^d\to\R$ and any convex set $A\subset B\subset \R^d$, we have $\partial w(A; B)=\partial w(A)$.
\item The volume of the subdifferentials has upper bound in terms of the non-divergence form difference operator as following.
\end{enumerate}

\begin{lemma}\label{lem:subdiff-ub}
For any $x\in B$ with $\partial u(x;B)\neq\emptyset$, we have 
\[
|\partial u(x;B)|\le (2 L_\omega u(x)/\kappa)_+^d.
\]
In particular, if $L_\omega u(x) \le \ell/2$ in $B$, then, with $\#B$ denoting the cardinality of $B$,
\[
|\partial u(B)|\le (\ell_+/\kappa)^d \#B.
\]
\end{lemma}
\begin{proof}
For $x\in B$ such that $\partial u(x;B)\neq\emptyset$, 
up to an affine translation, we may assume
\[
u(x)=0 \quad\text{ and }\quad 0\in\partial u(x;B).
\]
We will show that 
\[
\partial u(x;B)\subset \left[-(L_\omega u(x))_+/\kappa,( L_\omega u(x))_+/\kappa\right]^d.
\]
Indeed, for any $p\in\partial u(x;B)$, by the definition of the subdifferential set,
\[
u(x\pm e_i)-u(x)\ge\pm p\cdot e_i \quad \text{for all }i=1,\ldots,d.
\]
Moreover, since $0\in\partial u(x;B)$, we have $u(y)\ge u(x)$ for all $y\in\bar B$. Hence, by uniform ellipticity, we conclude that for every $i=1,\ldots, d$,
\[
 L_\omega u(x)=\sum_{e:|e|=1}\omega(x,x+e)[u(x+e)-u(x)]\ge \kappa |p\cdot e_i|.
\]
So, clearly, $L_\omega u(x) \leq 0$ implies $|\partial u(x;B)|=0$.

Let us now consider the case that $L_\omega u(x)>0$.
By scaling, we may assume that $L_\omega u(x)=1$. 
By the above inequality, $p\in[-1/\kappa, 1/\kappa]^d$ for every $p\in\partial u(x;B)$. Hence, $\partial u(x;B)\subset [-1/\kappa, 1/\kappa]^d$, and
the lemma follows.
\end{proof}

For $r>0$, let $\square_r:=\{x\in\Z^d: |x|_\infty< r/2\}$ denote the cube of side-length $r$ centered at the origin, and
\begin{equation}\label{eq:def-qn}
Q_n:=\square_{3^n}, \quad
R_n:=3^n.
\end{equation}
Note that $\#Q_n=3^{nd}$,where $\#A$ is the cardinality of a set $A$.
For each $n\in\N$, we divide $\Z^d$ into disjoint triadic cubes 
$\{y+Q_n: y\in 3^n\Z^d\}$,
among which we let $Q_n(x)$ denote the triadic cube that contains $x\in\Z^d$. 
\begin{definition} \label{def:parameter-s}
Assume that $\psi$ satisfies (A4).  Let $s\in\R$ and $B\subset\Z^d$. 
\begin{enumerate}
\item 
Recall \eqref{def:psi} and
define the sets of {\it super-solutions}
\[
S(s; B):=\left\{u\in\R^{\bar B}: L_\omega u(x) \leq s + \psi_\omega(x) \quad \forall x\in B\right\},
\]
\[
S^*(s; B):=\left\{u\in\R^{\bar B}: L_\omega u(x) \leq s - \psi_\omega(x) \quad \forall x\in B\right\}.
\]
Let the ``exact" solutions be
\begin{align*}
\mc E(s;B)&=\left\{u\in\R^{\bar B}: L_\omega u(x)  = s + \psi_\omega(x)\quad \forall x\in B\right\}, \\
\mc E^*(s;B)&=\left\{u\in\R^{\bar B}: L_\omega u(x) = s - \psi_\omega(x)\quad \forall x\in B\right\}.
\end{align*}
When $B=Q_n$, the above sets are written as $S_n(s), S_n^*(s), \mc E_n(s), \mc E_n^*(s)$, respectively.

\item 
Recall the cardinality notation $\#$ below \eqref{eq:def-qn}. 
Define
\[
\mu(s;B)=\frac1{\# B}\sup_{u\in S(s;B)}|\partial u(B)|, \quad
\mu^*(s;B)=\frac1{\# B}\sup_{u\in S^*(s;B)}|\partial u(B)|.
\]
When $B=Q_n$, the above quantities are written as $\mu_n(s), \mu^*_n(s)$, respectively.
Note that by Assumptions  (A1), (A3) and formula \eqref{eq:FK},  $\mu_n(s)$ and $\mu^*_n(s)$ are independent (under $\mb P$) of the environments $\{\omega(x):|x|>\tfrac{3^n}{2}+2\dpd\}$.
\end{enumerate}
\end{definition}
We remark that in the definition of $\mu_n(s)$, the set $S_n(s)$ can be replaced by $\mc E_n(s)$.
Indeed, if $u \in S_n(s)$, $v \in \mathcal{E}_n(s)$ with $v=u$ on $\partial Q_n$, then since $L_\omega u(x) \leq L_\omega v(x)$ in $Q_n$ it follows from the comparison principle that $v \leq u$ in $Q_n$. 
Therefore, $\partial u(Q_n)\subset \partial v(Q_n)$, 
and so 
\[
\sup_{u\in S_n(s)}|\partial u(Q_n)|=\sup_{u\in\mc E_n(s)}|\partial u(Q_n)|.
\]
Moreover, by Lemma~\ref{lem:subdiff-ub} and the definition of $\mu_n$, for $n\in\N$ and $s\in\R$,
\begin{equation}
\label{eq:190130}
\mu_n(s)\le 2^d \left[(2\norm{\psi}_\infty+s)_+/\kappa\right]^d.
\end{equation}
Similar inequality holds also for $\mu_n^*(s)$.
\begin{lemma}\label{lem:var}
 Assume (A1), (A2), (A4). Recall $\dpd$ in Assumption (A1).
\begin{enumerate}[(a)]
\item\label{item:var} 
 For all $m\ge4\dpd$,  $n\in\mathbb N$, $s\in\R$,
\[
\mb E[\mu_{m+n}^2(s)]\le 2R_n^{-d}\var[\mu_m(s)]+\mb E[\mu_m(s)]^2.
\]
\item \label{item:mono}
$\mb E[\mu_{n}(s)]$ and $\mb E[\mu_{n}(s)^2]$ are both non-increasing in $n$ for $n\ge 4\dpd$, and non-decreasing in $s$.
\item\label{item:limit-mu} Set $\mu_\infty(s):=\lim_{n\to\infty}\mb E[\mu_n(s)]$. Then $\varlimsup_{n\to\infty}\mu_n(s)=\mu_\infty$, $\mb P$-a.s., and 
\[
\lim_{n\to\infty}\mu_n(s)=\mu_\infty(s) \quad\text{ in }L^2(\mb P).
\]
\end{enumerate}
The same statements are true for $\mu_n^*(s)$ and $\mu_\infty^*(s):=\lim_{n\to\infty}\mb E[\mu_n^*(s)]$.
\end{lemma}
\begin{proof}
Clearly, both $\mb E[\mu_{n}(s)]$ and $\mb E[\mu_{n}(s)^2]$ are non-decreasing in $s$, since the set $S_n(s)$ in the definition of $\mu_n(s)$ is non-decreasing in $s$. 
The value of $s$ in $\mu_n(s)$ is irrelevant in the rest of the proof, and hence sometimes omitted.

Denote by $\{Q_m^i: 1\le i\le\# Q_n\}$ the collection of disjoint $m$-level sub-boxes   of $Q_{m+n}$. Let $\mu_m^{(i)}=\mu_m^{(i)}(s)=\sup_{u\in S(s;Q_m^i)}|\partial u(Q_m^i)|/\#Q_m$. 
Note that for any $u\in S_{m+n}(s)$ and $x\in Q_m^i$, $\partial u(x; Q_{m+n})\subset\partial u(x;Q_m^i)$ and so
\begin{equation}\label{eq:181030}
\mu_{m+n}\le \sum_{1\le i\le\#Q_n}\mu^{(i)}_m\big/\#Q_n.
\end{equation}
 Since $3^m\ge 4\dpd$ and by Assumption (A1), 
$\{\mu_m^{(i)}:1\le i\le\#Q_n\}$ 
is a 1-dependent sequence in the sense that $\mu_m^{(i)}$ is independent of $\mu_m^{(j)}$ as long as $Q_m^j$  is not adjacent to $Q_m^i$.  Hence, we have a decomposition of the index set 
\begin{equation}\label{eq:decomposition}
\{1,2,\ldots,\#Q_n\}=\Lambda_1\cup\Lambda_2
\end{equation}
with $\Lambda_1\cap\Lambda_2=\emptyset$ and $\#\Lambda_1,\#\Lambda_2>\#Q_n/3$ such that $\{\mu^{(i)}:i\in\Lambda_k\}, k=1,2,$ are both sets of independent random variables.
 Taking the first and second moments of both sides in \eqref{eq:181030}, we imply
$\mb E[\mu_{m+n}]\le \mb E[\mu_m]$ 
and
\begin{align*}
\mb E[\mu_{m+n}^2]
&\le 
\mb E\left[\bigg(
\sum_{1\le i\le\#Q_n}(\mu_m^{(i)}-\mb E[\mu_m])+\#Q_n\mb E[\mu_m]
\bigg)^2\right]\bigg/\#Q_n^2\\
&=
\mb E\left[\bigg(
\sum_{i\in\Lambda_1}(\mu_m^{(i)}-\mb E[\mu_m])
+\sum_{i\in\Lambda_2}(\mu_m^{(i)}-\mb E[\mu_m])
\bigg)^2\right]\bigg/\#Q_n^2+\mb E[\mu_m]^2
\\
&\le 
2[\#\Lambda_1 \var(\mu_m)+\#\Lambda_2 \var(\mu_m)]/\#Q_n^2+\mb E[\mu_m]^2\\
&=2\var(\mu_m)/\#Q_n+\mb E[\mu_m]^2.
\end{align*}
We thus obtain \eqref{item:var}, and that $\mb E[\mu_n]$ is non-increasing in $n$. Moreover, \eqref{item:var} also implies
\[
\mb E[\mu_{m+1}^2]-\mb E[\mu_m^2]\le (2\times 3^{-d}-1)\var[\mu_m]\le 0.
\]
Thus $\mb E[\mu_n^2]$ is also non-increasing in $n$ for $n\ge 4\dpd$.

 To prove \eqref{item:limit-mu}, 
sending first $n\to\infty$ and then $m\to\infty$ in 
\eqref{eq:181030}, and by the law of large numbers, we get $\varlimsup_{n\to\infty}\mu_n\le\mu_\infty$ almost surely. Moreover, by Fatou's lemma and the fact \eqref{eq:190130} that $\mu_n$'s are uniformly bounded from above, we get $\mb E[\varlimsup_{n\to\infty} \mu_n]\ge \mu_\infty$,  and so $\varlimsup_{n\to\infty} \mu_n=\mu_\infty$, $\mb P$-a.s.. To prove the $L^2$ convergence,  by Fatou's lemma we get $\lim_{n\to\infty}\mb E[\mu_n^2]\le\mu_\infty^2$.  Note that by \eqref{item:mono},  $\mb E[(\mu_n-\mu_\infty)^2]\le \mb E[\mu_n^2]-\mu_\infty^2$ for $n\ge 4\dpd$. Taking $n\to\infty$ on both sides we obtain $\lim_{n\to\infty}\mb E[(\mu_n-\mu_\infty)^2]\le 0$. 
Therefore, \eqref{item:limit-mu} is proved.
\end{proof}

By the ABP  inequality (Lemma~\ref{lem:abp}),  for all $u\in S_n(s)$ and $ -v\in S_n^*(s)$,
\begin{equation}\label{eq:u-mu-min}
-\min_{Q_n}u\le -\min_{\partial Q_n}u+CR_n^2\mu_n(s)^{1/d},
\end{equation}
\begin{equation}\label{eq:u-mu-max}
\max_{Q_n}v\le \max_{\partial Q_n}v+CR_n^2\mu_n^*(s)^{1/d}.
\end{equation}

\subsection{Lower bound of the convexity} \label{sec:lowerbound}
The goal of this subsection is to obtain lower bounds for the subdifferentials (Corollary~\ref{cor:mu-lowerb}). 
To this end,  we show a weak version (Lemma~\ref{lem:conv_rate_evfpvp}) of the quantitative result (Proposition~\ref{prop:conv_rate_evfpvp})  using an argument of Berger  \cite{B-per} that we learned from him through personal communications. 
Roughly speaking, due to the ergodicity of the environment viewed from the particle $(\bar\omega_n)_{n\in\N}$, the random walk behaves like a Brownian motion in the long run.  
 Hence, the homogenization error is rather flat in large scale where the flatness can be measured by subdifferential sets. Of course, how close the large scale random walk is to the Brownian motion depends on locally how ``good" the environment is. The 
 finite-range dependence of the environment enables us to say that with high probability, a large proportion of the environments are good.
Similar arguments can also be found in \cite[Theorem~1.4]{BCDG-18}.

\begin{lemma}\label{lem:conv_rate_evfpvp}
Assume (A1), (A2), (A3).
For any $\epsilon>0$, there exist constants  $C_\epsilon, m_\epsilon$ depending on $(\mb P,\epsilon,  \dpd)$ such that for all $m>m_\epsilon$  and $\mb P$-almost all $\omega$, the solution $\phi=\phi_m$ of \eqref{eq:corrector} in the cube $Q_m=\square_{3^m}$
satisfies 
\[
\mb P\left(\max_{Q_m}|\phi_m|/R_m^2\le \norm{\psi}_\infty\epsilon\right)\ge 1- e^{-C_\epsilon R_m^d}.
\] 
\end{lemma}
\begin{proof}
[Proof of Lemma~\ref{lem:conv_rate_evfpvp}]
Recall that $R_m=3^m$. 
Without loss of generality  we assume $\psi$ satisfies (A4). 
By Theorem~\ref{thm:QCLT} and the ergodic theorem, 
\[
\lim_{m\to\infty}\sum_{i=0}^{m-1}\psi(\bar\omega^i)/m=0
\quad\text{ $\mb P\times P_\omega$-a.s. and in $L^1(\mb P\times P_\omega)$. }
\]
 Hence,
for $\epsilon>0$, there exists $n=n(\epsilon,\mb P,  \dpd)>\dpd$ (which without loss of generality we can assume is larger than $\epsilon^{-1}$) such that 
\[
\mb P\left( \Abs{\frac1nE_\omega\left[\sum_{i=0}^{n-1} \psi(\bar\omega^i)\right]}\ge \epsilon\right)<\epsilon.\]
We say that a point $x\in\Z^d$ is {\it ($\epsilon$-)good} (and otherwise {\it bad}) if $\Abs{E_\omega^x[\sum_{i=0}^{n-1} \psi(\bar\omega^i)]/n}\le \epsilon$. Note that the event ``$x$ is bad" only depends on environments $\{\omega(y): y \in B_n(x)\}$ in ball $B_n(x)$.
Since  we assume (A1) and (A3),  the random varaibles $\{\mathbbm{1}_{x\text{ is bad}}: x\in\Z^d\}$ have a 
 range $2n+4\dpd<6n$ of dependence.
 When  $R_m>12n$, set $Q_{m,n}:=\{x\in Q_m: \dist(x,\partial Q_m)>6n\}$.  
Since one can decompose $Q_{m,n}$ into $(6n)^d$ subsets such that each subset consits of roughly $(R_m-12n)^d/(6n)^d$ points for which the random variables $\mathbbm 1_{x\text{ is bad}}$ are i.i.d., then Cram\'er's Theorem implies that 
\[
\mb P\left(\sum_{x\in Q_{m,n}}\mathbbm 1_{x\text{ is bad}}/\#Q_{m,n}\ge 2\epsilon \right)
\le 
Cn^de^{-c_\epsilon(R_{m}-2n)^d}\le Cn^de^{-C_\epsilon R_m^d}.
\]
We will show that for $R_m>n(\epsilon)^2$,
\begin{equation}\label{eq:181010}
\mb P(\mu_{m}(-2\epsilon)\le C\epsilon)\ge 1-Cn^de^{-C_\epsilon R_m^d}.
\end{equation}
Indeed, if $x\in Q_{m,n}$ is good, then for any $u\in\mc E_m(-2\epsilon)$,
\begin{align*}
 E^x_\omega[u(x)-u(X_n)]
&= E^x_\omega\left[\sum_{i=0}^{n-1}\left( - \psi(\bar\omega^i)+2\epsilon\right)\right]\\
&\ge 
-\epsilon n+2\epsilon n>0
\end{align*}
which implies $\partial u(x;Q_m)=\emptyset$. Hence, using the fact (Lemma~\ref{lem:subdiff-ub}) that $|\partial u(x;Q_m)|\le C$,  we have with $\mb P$-probability at least $1-n^de^{-C_\epsilon R_m^d}$,
\[
|\partial u(Q_m)|
\le 
C\#(Q_m\setminus Q_{m,n})+C\#\text{bad points in }Q_{m,n}\le  CnR_m^{d-1}+C\epsilon R_m^d
\]
and so $\mu_{m}(-2\epsilon)\le CnR_m^{-1}+C\epsilon\le Cn^{-1}+C\epsilon  \le C\epsilon$. This completes the proof of inequality \eqref{eq:181010}.

Let $\tau=\min\{k\ge 0: X_k\notin Q_m\}$ and set $v(x):=\phi_m(x)+2\epsilon E_\omega^x[\tau]\in\mc E_m(-2\epsilon)$. By \eqref{eq:u-mu-min} and \eqref{eq:181010},
$\mb P(\min_{Q_m}v/R_m^2\ge -C\epsilon^{1/d})\ge 1-Cn^de^{-C_\epsilon R_m^d}$. 
Recalling \eqref{eq:exittime} we have that $E_\omega^x[\tau] \le d R_m^2$. 
So $\min_{Q_m}\phi_m/R_m^2\ge \min_{Q_m}v/R_m^2-C\epsilon$ and
\[
\mb P\left(\min_{Q_m}\phi_m/R_m^2\ge -2C\epsilon^{1/d} \right)\ge 1-Cn^de^{-C_\epsilon R_m^d}.
\]
Similarly $\mb P(\max_{Q_m}\phi_m/R_m^2\le 2C\epsilon^{1/d})\ge 1-n^de^{-C_\epsilon R_m^d}$. The lemma is proved.
\end{proof}
\begin{remark}\label{rem:subquadratic}
It follows immediately from Lemma~\ref{lem:conv_rate_evfpvp} that
\begin{equation}\label{eq: a.s.-homo}
\lim_{m\to\infty}\max_{Q_m}|\phi_m|/R_m^2= 0
\end{equation}
$\mb P$-almost surely and in $L^p(\mb P)$ for any $p>0$. 
\end{remark}

By Lemma~\ref{lem:conv_rate_evfpvp},  the homogenization error is uniformly flat in large scale. Consequently, adding a bit of convexity to the random solution will bend the corresponding effective solution like a paraboloid. 

\begin{corollary}\label{cor:mu-lowerb}
Assume (A1), (A2), (A4).
For any $s>0$, there exists $C=C(d)$ such that
\[
\mu_\infty(s)\ge Cs^d \quad\text{and}\quad 
\mu_\infty^*(s)\ge Cs^d.
\]
\end{corollary}

\begin{proof}
Let $\phi_n$ denote the solution of the Dirichlet problem \eqref{eq:corrector} in $Q_n$.

Let $u_n(x)=\phi_n(x)-sE_\omega^x[\tau]$, where $\tau=\tau(Q_n)$ denotes the exit time of the random walk from cube $Q_n$. 
Notice that $E_\omega^x[\tau]=E_\omega^x[|X_\tau|^2]-|x|^2$, and that $u_n\in\mc E_n(s)$ with $u_n=0$ in $\partial Q_n$. 
Furthermore, $\min_{Q_n}u_{n}\le\max_{Q_n}\phi_n-s E_\omega^0[\tau]$.
Hence,
\[
\mu_n(s)^{1/d}
\stackrel{\eqref{eq:u-mu-min}}{\ge} -C\min_{Q_n}u_n/R_n^2
\ge 
-C\max_{Q_n}\phi_n/R_n^2+Cs E_\omega^0[\tau(Q_{n})]/R_n^2.
\]
Taking $n\to\infty$, using \eqref{eq: a.s.-homo}, Lemma~\ref{lem:var}\eqref{item:limit-mu}, and the fact that 
\[
\lim_{n\to\infty}E_\omega^0[\tau(Q_{n})]/R_n^2\ge \lim_{n\to\infty}E_\omega^0[\tau(B_{R_n})]/R_n^2=1,
\]
 we obtain $\mu_\infty(s)\ge Cs^d$. The second inequality can be proved similarly.
\end{proof}

As a consequence of the lower bound,  we can deduce that the sequences $\mu_n, \mu_n^*$ will ``stabilize" at some point.
\begin{proposition}\label{prop:new}
Assume (A1), (A2), (A4). For any $\varepsilon>0$, there are constants $N, c_\varepsilon>0$ depending on $(\varepsilon,d,\kappa,\dpd)$ such that for any $\ell\ge N$ and $n\ge 2\ell$, there exists $m\le n-\ell$ such that for all $0\le j\le \ell$, 
\begin{align*}
\mb E\left[\left(\mu_{m+j}(e^{-c_\varepsilon n})-\mb E[\mu_{m+\ell}(e^{-c_\varepsilon n})]\right)^2\right]
&\le 
\varepsilon\mb E\big[\mu_{m+\ell}(e^{-c_\varepsilon n})\big]^2,\\
\mb E\left[\big(\mu^*_{m+j}(e^{-c_\varepsilon n})-\mb E[\mu^*_{m+\ell}(e^{-c_\varepsilon n})]\big)^2\right]&\le 
\varepsilon\mb E\big[\mu^*_{m+\ell}(e^{-c_\varepsilon n})\big]^2.
\end{align*}
\end{proposition}
\begin{proof}
Let $N\in\N$ be a constant to be determined.  Set, for $i\ge 0, n\ge 2N$,
\[
S_{i,n}=\ln\left(
\mb E[\mu_{2\ell i}^2(\delta^{-n/(4d)})]\cdot
\mb E[\mu_{2\ell i}^{*2}(\delta^{-n/(4d)})]\big/ C
\right),
\]
where $\delta:=1+\tfrac\varepsilon 2$. By Corollary~\ref{cor:mu-lowerb} and \eqref{eq:190130},
\[
\sum_{i=1}^{n-1}S_{i+1,n}-S_{i,n}=S_{n,n}-S_{1,n}\ge -\tfrac n 2\ln\delta-C.
\]
Hence there exists $1\le k<n$ such that (note that $N$ is sufficiently big and $n>N$),
\[
S_{k+1,n}-S_{k,n}\ge 
 -\tfrac{n}{2(n-1)}\ln\delta-\tfrac{C}{n-1}\ge -\ln\delta.
\]
For simplicity of notations, we write $\mu_i:=\mu_i(\delta^{-n/(4d)}), \mu_i^*:=\mu_i^*(\delta^{-n/(4d)})$.  Then the above inequality implies $\mb E[\mu_{2\ell k}^2]\cdot \mb E[\mu_{2\ell k}^{*2}]\le \delta \mb E[\mu_{2\ell(k+1)}^2]\cdot \mb E[\mu_{2\ell(k+1)}^{*2}]$.
By Lemma~\ref{lem:var}\eqref{item:mono}, since $\mb E[\mu_i^2], \mb E[\mu_i^{*2}]$ are non-increasing in $i$ for $i\ge N>4\dpd$,  we get
\begin{equation}
\label{eq:stabilize}
\mb E[\mu_{2\ell k}^2]\le \delta \mb E[\mu_{2\ell(k+1)}^2], \quad
\mb E[\mu_{2\ell k}^{*2}]\le \delta \mb E[\mu_{2\ell(k+1)}^{*2}].
\end{equation}
Now set $m=2\ell k\le n-2\ell$.  We have, for $0\le j\le\ell$,
\begin{align*}
\mb E[(\mu_{m+j}-\mb E[\mu_{m+\ell}])^2]
&\stackrel{Lemma~\ref{lem:var}\eqref{item:mono}}{\le}
\mb E[\mu^2_m]-\mb E[\mu_{m+\ell}]^2\\
&\stackrel{\eqref{eq:stabilize}}{\le}
\delta\mb E[\mu_{m+2\ell}^2]-\mb E[\mu_{m+\ell}]^2\\
&\stackrel{Lemma~\ref{lem:var}\eqref{item:var}}{\le}
\delta\left(2R_{\ell}^{-d}\var(\mu_{m+\ell})+\mb E[\mu_{m+\ell}]^2\right)-\mb E[\mu_{m+\ell}]^2.
\end{align*}
In particular, for $j=\ell$, this inequality yields
\[
\var(\mu_{m+\ell})\le \frac{\delta-1}{1-2\delta R_N^{-d}}\mb E[\mu_{m+\ell}]^2
\le 2(\delta-1)\mb E[\mu_{m+\ell}]^2.
\]
Finally, using the above two inequalities, we get, for $0\le j\le\ell$,
\[
\mb E[(\mu_{m+j}-\mb E[\mu_{m+\ell}])^2]
\le (\delta-1)(1+4\delta R_\ell^{-d})\mb E[\mu_{m+\ell}]^2
\le 2(\delta-1)\mb E[\mu_{m+\ell}]^2.
\]
Similarly, we obtain $\mb E[(\mu_{m+j}^*-\mb E[\mu_{m+\ell}^*])^2]\le 2(\delta-1)\mb E[\mu_{m+\ell}^*]^2$, $\forall 0\le j\le \ell$.
\end{proof}


\subsection{Upper bound of the convexity} \label{sec:upperbound}
The goal of this subsection is to obtain an upper bound (Theorem~\ref{thm:190225}) for $\mu_n(s)$ and $\mu_n^*(s)$   when the convexity of solutions in some smaller sub-cubes are stabilized. 

%

Theorem~\ref{thm:190225} will play a crucial role in establishing exponential upper bounds for $\mu_n(0), \mu_n^*(0)$.  It states that,  in a fixed environment,  if  the upper and lower bounds of the convexity of the perturbed solutions are comparable,  then the convexity has an algebraic bound in terms of $s$.
It does the job of Lemmas 3.1, 3.2, 3.3, Corollary 3.4 and part of the proof of Lemma 4.1 in \cite{AS-14}.
The differences between our proof and that of \cite{AS-14} are summarized in Remark~\ref{rmk:convexity}.  Note that these are technical improvements and simplifications. We strongly rely on the strategy of \cite{AS-14} where the bounds of the convexity are used to control the homogenization error.
\begin{theorem}\label{thm:190225}
Let $m\ge 0, s,a>0,\omega\in\Omega$. There exist constants $n_0\in\N$ and $\lambda\in(0,1)$ depending on $(d,\kappa)$ such that, assuming that for some $n\ge 2n_0$,
\begin{enumerate}[(i)]
\item\label{eq:190225-1}
$\mu_{m+n-n_0+1}(s)+\mu_{m+n-n_0+1}^*(s)\le 10a$;
\item\label{item:mostpoint} 
 there are non-negative functions $u\in \mc E_{m+n}(s), u^*\in \mc E_{m+n}^*(s)$ with $\min_{Q_{m+n-n_0}}u=\min_{Q_{m+n-n_0}}u^*=0$ and (Recall the notations $\#$ and $Q_n(x)$ under \eqref{eq:def-qn}.)
at least $(1-\lambda)\# Q_{m+n}$ points $x\in Q_{m+n}$ satisfy
\[
|\partial u(Q_m(x);Q_{m+n})|+|\partial u^*(Q_m(x);Q_{m+n})|\ge a \#Q_m,
\]
\end{enumerate}
then  $a\le Cs^d$.
\end{theorem}
\begin{lemma}\label{lem:190225}
 Let $m\ge 0,n\in\N, s,a>0$, and  $0\le k<n$. Assume that 
\[
\mu_{m+k+1}(s)+\mu_{m+k+1}^*(s)\le a.
\]
Then for any non-negative functions $u\in \mc E_{m+n}(s), u^*\in \mc E_{m+n}^*(s)$ with $\min_{Q_{m+k}}u=\min_{Q_{m+k}}u^*=0$, we have
\[
\max_{Q_{m+k}}(u+u^*)\le Ca^{1/d}R_{m+k}^2.
\]
\end{lemma}
\begin{proof}
Let $g, g^*: \bar Q_{m+k+1}\to\R$ be functions that solve
\[
\left\{
\begin{array}{lr}
L_\omega g=L_\omega g^*=0 &\text{ in }Q_{m+k+1},\\
g=u,\, g^*=u^* &\text{ on }\partial Q_{m+k+1}. 
\end{array}
\right.
\]
Note that $g,g^*$ are non-negative. Let $\tilde u=u-g\in \mc E_{m+k+1}(s)$. Then $\tilde{u}|_{\partial Q_{m+k+1}}=0$. By assumption, there exists $x_0\in Q_{m+k}$ with $u(x_0)=0$. By the Harnack inequality for non-divergence form difference operators (see \cite[Theorem~3.1]{KT-96}, and  also \cite[A.1.3]{G-12} for more detailed proof), 
\[
\max_{Q_{m+k}}g\le Cg(x_0)=-C\tilde{u}(x_0)\stackrel{\eqref{eq:u-mu-min}}{\le} Ca^{1/d}R_{m+k}^2.
\]
Similarly, we get $\max_{Q_{m+k}}g^*\le Ca^{1/d}R_{m+k}^2$. Hence
\[
\max_{Q_{m+k}}(g+g^*)\le Ca^{1/d}R_{m+k}^2.
\]
Setting $v=u+u^*-(g+g^*)$, we have $L_\omega v=2s$ in $Q_{m+k+1}$ and $v|_{\partial Q_{m+k+1}}=0$. Thus, in $Q_{m+k+1}$, we have $v\le 0$ and so  $u+u^*\le g+g^*$. The lemma follows.
\end{proof}
\begin{proof}[Proof of Theorem~\ref{thm:190225}]  
Recall the definition of $\breve{u}$ in \eqref{eq:def-envelope}. 
 Let $\breve{u}(x)=\breve{u}(x;Q_{m+n})$ and
\[
S:=\left\{x\in\R^d: \breve{u}(x)\le ha^{1/d}R_{m+n}^2\right\}.
\]
Since $\breve{u}$ is convex, $S$ is a convex set. 
Let $n_0\in\N$ be a constant to be determined.

First, we will show via contradiction that for $h:=R_{n_0}^{1/d-1}$ and $n\ge2 n_0$,
\begin{equation}\label{eq:190225-2}
\min_{\partial Q_{m+n}}u+\min_{\partial Q_{m+n}}u^*
\ge 
ha^{1/d}R_{m+n}^2.
\end{equation}
Indeed, if \eqref{eq:190225-2} fails, then there exists $x_1\in \partial Q_{m+n}\cap S$. Moreover, 
setting $n_1:=n-n_0 \ge n_0$,
by Lemma~\ref{lem:190225}, 
$\max_{Q_{m+n_1}}(u+u^*)\le Ca^{1/d}R_{m+n_1}^2=Ca^{1/d}R_{m+n}^2R_{n_0}^{-2}$. Thus $Q_{m+n_1}\subset S$ if $n_0$ is large enough. 
Hence, $S$ contains the convex hull of $x_1$ and $Q_{m+n_1}$. 
In particular, setting $e':=x_1/|x_1|$, $S$ contains the cone $\mc C$ with vertex $x_1$ and base $\left\{x: x\cdot e'=0, |x|\le R_{m+n_1}/2\right\}$. 
Now let
\[
\mc C'=\tfrac{x_1}2+\tfrac 12 (\mc C-\tfrac{x_1}2)\subset \mc C.
\]
\begin{figure}[ht]
\centering
 \includegraphics[width=0.6\textwidth]{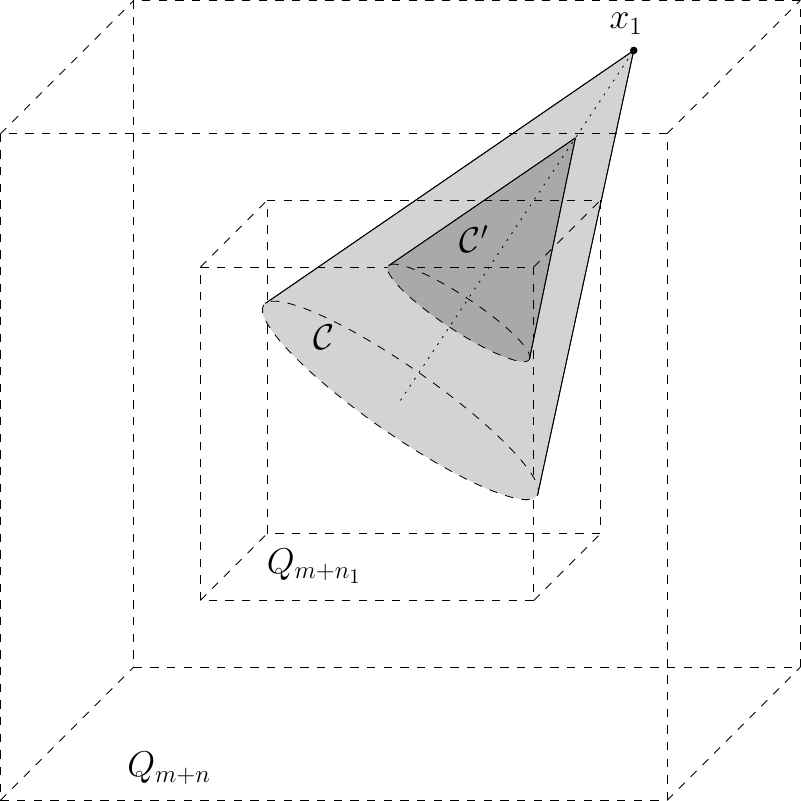}
\caption{Graphical description of the cones $\mc C, \mc C'$.}
\end{figure}
 Note that every point in $\mc C'$ is of distance at least $R_{m+n_1}/16$ away from the surface $\partial \mc C$ of $\mc C$. Hence, taking $n_0$ large enough, we have (Recall $Q_n(x)$ under \eqref{eq:def-qn}.)
\[
\mc C_m:=\bigcup_{x\in\mc C'\cap\Z^d}Q_m(x)\subset \mc C\cap\Z^d
\]
and that every point in $\mc C_m$ is at least of distance $R_{m+n_1}/16-\sqrt d R_m\ge R_{m+n_1}/32$ away from $\partial\mc C$. Notice also that $\dist(x_1,\mc C_m)\ge R_{m+n}/8-R_m\ge R_{m+n}/16$ and similarly $\dist(0,\mc C_m)\ge R_{m+n}/16$.
For any $e\in\R^d$ with $|e|=1$ and any $p\in \partial u(\mc C_m;Q_{m+n})$, say, $p\in \partial u(y;Q_{m+n})=\partial\breve{u}(y;Q_{m+n})$ for $y\in \mc C_m$, then $y\pm e R_{m+n_1}/32\in\mc C\subset S$, and hence,
\[
p\cdot (\pm e R_{m+n_1}/32)\le \breve{u}(y\pm e R_{m+n_1}/32)-\breve{u}(y)\le 
ha^{1/d}R_{m+n}^2.
\]
Moreover, $y\pm e'R_{m+n}/16\in\mc C$, and similar argument yields
\[
 p\cdot(\pm e' R_{m+n}/16) \le ha^{1/d}R_{m+n}^2.
\]
We conclude that $|p\cdot e|\le cha^{1/d}R_{m+n}R_{n_0}$ for all $|e|=1$, and $|p\cdot e'|\le cha^{1/d}R_{m+n}$. In other words, $\partial u(\mc C_m;Q_{m+n})$ is contained in a cylinder with height $cha^{1/d}R_{m+n}$ and base radius $cha^{1/d}R_{m+n}R_{n_0}$.  Hence
\[
|\partial u(\mc C_m;Q_{m+n})|\le Cah^d R_{m+n}^d R_{n_0}^{d-1}
\le CaR_{n_0}^{-1}\#\mc C_m,
\]
where in the last inequality we used $h=R_{n_0}^{1/d-1}$ and the fact that $\#\mc C_m\ge |\mc C'|\ge c|\mc C|\ge CR_{m+n}^d R_{n_0}^{-(d-1)}$.
Similar arguments yield the same upper bound for $u^*$.  Thus
\begin{equation}\label{eq:190225-3}
|\partial u(\mc C_m;Q_{m+n})|+|\partial u^*(\mc C_m;Q_{m+n})|
\le CaR_{n_0}^{-1}\#\mc C_m.
\end{equation}
On the other hand,  by \eqref{item:mostpoint},  choosing $\lambda=c_0R_{n_0}^{-(d-1)}$ where $c_0(d,\kappa)$ is a small constant so that
\[
(1-\lambda)\#Q_{m+n}\ge \#Q_{m+n}-\tfrac12\#\mc C_m,
\]
we get at least half of the points in $\mc C_m$ satisfying the inequality in \eqref{item:mostpoint}. Hence,
\begin{equation}\label{eq:190225-4}
|\partial u(\mc C_m;Q_{m+n})|+|\partial u^*(\mc C_m;Q_{m+n})|\ge \tfrac{1}{2}a\#\mc C_m.
\end{equation}
Combining \eqref{eq:190225-3} and \eqref{eq:190225-4}, we get  $Ca\le R_{n_0}^{-1}a$
which is absurd if $n_0$ is large enough. 
Display \eqref{eq:190225-2} is proved.

Finally, 
since $L_\omega(u+u^*)=2s$, by \eqref{eq:190225-2}
and \eqref{eq:u-mu-min}, 
\[
\min_{Q_{m+n}}(u+u^*)\ge ha^{1/d}R_{m+n}^2-CsR_{m+n}^2.
\]
On the other hand, by Lemma~\ref{lem:190225}, 
\[
\min_{Q_{m+n}}(u+u^*)\le \max_{Q_{m+n_1}}(u+u^*)\le Ca^{1/d}R_{m+n_1}^2.
\]
Therefore, 
$ha^{1/d}-Cs\le Ca^{1/d}R_{n_0}^{-2}$. 
Recalling $h=R_{n_0}^{1/d-1}$, for $n_0$ sufficiently large, we get $Ca^{1/d}\le s$.  Our proof is complete.
\end{proof}

\begin{remark}\label{rmk:convexity}
A key step in the above proof is to obtain \eqref{eq:190225-2}.
For this, we borrow some ideas in \cite[Lemma 3.1]{AS-14}, which states that if a function in a cube is quite convex locally at all points, then it either bends up on the whole boundary or bends up over a strip. 
See also earlier works \cite{Ca-90, M-15}.
Note that a key difference here between our Theorem \ref{thm:190225} and \cite[Lemma 3.1]{AS-14} is that we do not require the function to be quite locally convex at all points but only at a large portion of points (see assumptions \eqref{item:mostpoint} of Theorem~\ref{thm:190225}).

Besides, the proof  of \eqref{eq:190225-2} is done directly through the analysis of the convex set $S$, the cones $\mc C, \mc C'$, and the subdifferential set  $\partial u(\mc C_m;Q_{m+n})$.
Because of the clear geometry of the cones,  we do not need to use John's lemma (see \cite[Lemma 3.23]{LMT-17}), which says that
for any closed convex set $S$ with nonempty interior, there exists an affine map $\phi:\R^d\to\R^d$ such that $\bar{\B}_1\subset\phi(S)\subset \bar \B_d$.

Finally, unlike  \cite[Lemma 3.1]{AS-14},  the second situation that $u$ might bend up over a strip instead of the whole boundary is ruled out thanks to assumption \eqref{eq:190225-1}.
\end{remark}

\subsection{Quantification of ergodicity via the concentration of convexity}

This subsection is devoted to the proof of Proposition~\ref{prop:conv_rate_evfpvp}.  

As a key step,  we obtain an exponential decay (Theorem~\ref{thm:conv_rate_subdiff}) for the second moments of $\mu_n(0)$ and $\mu_n^*(0)$.  
To this end, we first use  Proposition~\ref{prop:new} to deduce that with positive probability, the convexity of the perturbed solutions stabilize at certain scale. This implies that the conditions of Theorem~\ref{thm:190225} are satisfied in certain environments. Then, by Theorem~\ref{thm:190225}, we get an exponential upper bound for the second moment of the convexity in triadic cubes.
\begin{theorem}
\label{thm:conv_rate_subdiff}
 Assume (A1), (A2), (A4).
There exist constants $C, c$ depending on $(d,\kappa,\dpd)$ such that, for all $k\ge 0$, 
\[
\mb E[\mu_k(0)^2+\mu_k^*(0)^2]\le Ce^{-ck}.
\]
\end{theorem}

\begin{proof}
Recall $N$ in Proposition~\ref{prop:new}. Redefine $n_0$ in Theorem~\ref{thm:190225} so that $n_0>N$.  By \eqref{eq:190130}, we only need to prove Theorem~\ref{thm:conv_rate_subdiff} for all $k\ge 4n_0$.

Let $\varepsilon=\varepsilon(n_0,d,\kappa,\dpd)>0$ be a constant to be determined later, and let $c_\varepsilon$ be as in Proposition~\ref{prop:new}.  We set $\ell=2n_0, s=e^{-c_\varepsilon k}$, and write $\mu_n(s), \mu_n^*(s)$ as $\mu_{n,1},\mu_{n,2}$, respectively.  
For $k\ge 4n_0$, by Proposition~\ref{prop:new}, there exists $m\le k-\ell$ such that 
\begin{equation}\label{eq:stabilize-bound}
\mb E\left[|\mu_{m+j,i}-\mb E[\mu_{m+\ell,i}]|^2\right]\le\varepsilon\mb E[\mu_{m+\ell,i}]^2, \quad\text{for all } 0\le j\le\ell, i=1,2.
\end{equation}
Let $I=\{1,\ldots, 3^{2dn_0}\}$ and denote the collection of disjoint $m$-level subcubes of $Q_{m+2n_0}$ by $\{Q_{m}^j:j\in I\}$. We let $\mu_{m,1}^{(j)}=\mu( s; Q_{m}^j)$ and $\mu_{m,2}^{(j)}=\mu^*(s;Q_{m}^j)$ for $j\in I$, and set $a_i=\mb E[\mu_{m+\ell,i}]$, $i=1,2$. Define the event
\[
A=\{\omega\in\Omega:|\mu_{m+\ell,i}-a_i|\le\varepsilon^{1/4}a_i, 
|\mu_{m,i}^{(j)}-a_i|\le\varepsilon^{1/4}a_i \text{ for all }j\in I, i=1,2
\}.
\]

\noindent{\bf Step 1.}  First, we claim that $\mb P(A)>0$. Indeed,  by Chebyshev's inequality and \eqref{eq:stabilize-bound},  $\mb P(|\mu_{m,i}^{(j)}-a_i|\ge \varepsilon^{1/4} a_i)\le \mb E[(\mu_{m,i}-a_i)^2]/(\varepsilon^{1/4} a_i)^2\le \varepsilon^{1/2}$. Similarly, $\mb P(|\mu_{m+\ell, i}-a_i|\ge \varepsilon^{1/4} a_i)\le  \varepsilon^{1/2}$ for $i=1,2$. Hence $\mb P(A^c)\le 2(3^{2dn_0}\varepsilon^{1/2}+\varepsilon^{1/2})<1$ if $\varepsilon>0$ is chosen to be sufficiently small.  The claim is proved.

\noindent{\bf Step 2.} From now on we let $\omega\in A$ be a fixed environment. Setting 
\begin{equation}
\label{eq:def-a}
a=(1-\varepsilon^{1/8})(a_1+a_2),
\end{equation}
we will verify that conditions \eqref{eq:190225-1}\eqref{item:mostpoint} of Theorem~\ref{thm:190225} are satisfied for $n=2n_0$.  
To verify \eqref{eq:190225-1}, by \eqref{eq:181030} and the definition of $A$, we have
$\mu_{m+n_0+1,i}\le (1+\varepsilon^{1/4})a_i$ for $i=1,2$.
Hence \eqref{eq:190225-1} is satisfied.  To verify \eqref{item:mostpoint},  recall the notation $Q_n(x)$ under \eqref{eq:def-qn}.We pick functions $u\in \mc E_{m+2n_0}(s)$, $u^*\in \mc E^*_{m+2n_0}(s)$ such that 
\[
\frac{|\partial u(Q_{m+2n_0})|}{\#Q_{m+2n_0}}
\ge (1-2\varepsilon^{1/4})a_1, \quad \frac{|\partial u^*(Q_{m+2n_0})|}{\#Q_{m+2n_0}}
\ge (1-2\varepsilon^{1/4})a_2.
\]
Since $|\partial u(Q_m;Q_{m+2n_0})|\le |\partial u(Q_m)|\le \mu_{m,1}\#Q_m$, taking $\varepsilon>0$ sufficiently small,
\begin{align*}
\frac{|\partial u(Q_m;Q_{m+2n_0})|}{\#Q_{m+2n_0}}
&=
\frac{|\partial u(Q_{m+2n_0})|}{\#Q_{m+2n_0}}-
\frac{1}{\#Q_{m+2n_0}}\sum_{x\in Q_{m+2n_0}\setminus Q_m}\frac{|\partial u(Q_m(x);Q_{m+2n_0})|}{\#Q_m}\\
&\ge 
(1-2\varepsilon^{1/4})a_1-(1-3^{-2n_0d})(1+\varepsilon^{1/4})a_1>0.
\end{align*}
Hence, up to an affine transformation, we can assume  $\min_{Q_m}u=\min_{Q_{m+2n_0}}u=0$.
Furthermore, let $\Lambda=\{x\in Q_{m+2n_0}: |\partial u(Q_m(x);Q_{m+2n_0})|\ge (1-\varepsilon^{1/8})a_1\#Q_m\}$ and $p:=\#\Lambda/\#Q_{m+2n_0}$.  
Then,
\begin{align*}
(1-2\varepsilon^{1/4})a_1
&\le 
\frac{|\partial u(Q_{m+2n_0})|}{\#Q_{m+2n_0}}= \frac{1}{\#Q_{m+2n_0}}\sum_{x\in Q_{m+2n_0}}\frac{|\partial u(Q_m(x);Q_{m+2n_0})|}{\#Q_m}\\
&\le (1-p)(1-\varepsilon^{1/8})a_1+p(1+\varepsilon^{1/4})a_1,
\end{align*}
which implies $p\ge (1-2\varepsilon^{1/8})/(1+\varepsilon^{1/8})=:p_\varepsilon$.  Similar inequality holds for $u^*$. Hence, taking $\varepsilon$ to be small enough,  we have $2(1-p_\varepsilon)\le\lambda$ and therefore \eqref{item:mostpoint} of Theorem~\ref{thm:190225} is also satisfied for $n=2n_0$.

\noindent{\bf Step 3.} Recall \eqref{eq:def-a}. By Theorem~\ref{thm:190225},  we get $a_1+a_2\le Cs^d$. Further, by \eqref{eq:stabilize-bound},
\[
\mb E[\mu_{m+\ell,i}^2]=\var(\mu_{m+\ell,i})+a_i^2
\le 
(1+\varepsilon)a_i^2, \quad i=1,2.
\]
Therefore,
\begin{align*}
\mb E[\mu_k(0)^2+\mu_k^*(0)^2]\le
\mb E[\mu_{m+\ell,1}^2+\mu_{m+\ell,2}^2]\le C(a_1^2+a_2^2)\le Cs^{2d}=Ce^{-ck},
\end{align*}
where in the first inequality we used Lemma~\ref{lem:var}\eqref{item:mono}.
\end{proof}

Using \eqref{eq:181030} and the fact that $\mu_n(0)$ are uniformly bounded, we can obtain the following improved concentration bound, which is similar to \cite[Corollary~2.10]{AS-14}.

\begin{corollary}\label{cor:mu_conc_bound}
Assume (A1), (A2), (A4).
For any $p\in(0,d)$, there exists a constant $\alpha=\alpha(d,\kappa, p,\dpd)>0$ such that for all $t\ge 1$ and $n\in\N$,  
\[
\mb P(\mu_n(0)+\mu_n^*(0)\ge C3^{-\alpha n}t)
\le
4\exp(-c(t-1)^23^{np}).
\]
\end{corollary}
\begin{proof}
Adjusting the value of $C$ if necessary, it suffices to consider $n$ which is sufficiently large.
Let $\theta\in(0,1)$ to be determined and write $n=n_1+n_2$ with $n_1=\floor{\theta n}$.  Note that by \eqref{eq:190130}, $\mu_n(0)\le C$. 
Set $\bar\mu_n=\mb E[\mu_n(0)]$.  By Theorem~\ref{thm:conv_rate_subdiff} there are constants $C_1, c_2$ depending on $(d,\kappa,\dpd)$ such that $\bar\mu_{n}\le C_13^{-c_2n}$. 
Recall that we have a decomposition $\{1,\ldots,\#Q_{n_2}\}=\Lambda_1\cup\Lambda_2$ as in \eqref{eq:decomposition}.
By \eqref{eq:190130}, \eqref{eq:181030},
\begin{align*}
\mb P(\mu_n(0)\ge 2C_13^{-c_2n_1}t)
&\le 
\mb P\left(
\sum_{i\in\Lambda_1}\mu_{n_1}^{(i)}\big/\#\Lambda_1+\sum_{j\in\Lambda_2}\mu_{n_1}^{(j)}\big/\#\Lambda_2
\ge 2C_13^{-c_2n_1}t
\right)\\
&\le 
\sum_{j=1}^2\mb P\bigg(
\sum_{i\in\Lambda_j}[\mu_{n_1}^{(i)}-\bar\mu_{n_1}]\big/\#\Lambda_j\ge (t-1)C_13^{-c_2n_1} \bigg)
\\
&\le 
2\exp\left(
-C\#Q_{n_2}(t-1)^23^{-2c_2n_1}
\right),
\end{align*}
where we used Hoeffding's inequality (See, e.g., \cite[Theorem~2.8]{BLM-13}) in the last inequality.
Since $\# Q_{n_2}=3^{n_2d}$, we conclude that
\[
\mb P(\mu_n(0)\ge 2C_13^{-c_2\theta n}t)\le 
2\exp\left(
-C(t-1)^23^{[(1-\theta)d-2c_2\theta]n}
\right).
\]
Similar inequality holds for $\mu_n^*(0)$. 
The corollary follows by noticing that 
\[
\mb P(\mu_n(0)+\mu_n^*(0)\ge 2C3^{-\alpha n}t)\le \mb P(\mu_n(0)\ge C3^{-\alpha n}t)+ \mb P(\mu_n(0)\ge C3^{-\alpha n}t)\]
and taking $\theta=\theta(p)$ appropriately.
\end{proof}

\begin{proof}
[Proof of Proposition~\ref{prop:conv_rate_evfpvp}]
Without loss of generality, assume that $\psi$ satisfies (A4).
For $n\in\N$, let $\phi_n \in\mc E_n(0)$ denote the solution of the Dirichlet problem \eqref{eq:corrector} in $Q_n$. Note that $-\phi_n \in \mc E^*_n(0)$. By \eqref{eq:u-mu-min} and \eqref{eq:u-mu-max}, we have
\[
\tfrac{1}{R_n^2}\max_{Q_n}|\phi_n|\le C[\mu_n(0)+\mu_n^*(0)]^{1/d}.
\]
By Corollary~\ref{cor:mu_conc_bound}, for $p\in(0,d)$, there is $\alpha=\alpha(d,\kappa,p)>0$ such that 
\begin{equation}\label{eq:190130-2}
\mb P\big(
\tfrac{1}{R_n^2}\max_{Q_n}|\phi_n|
\ge CR_n^{-\alpha}
\big)
\le 
C\exp(-cR_n^p).
\end{equation}
To obtain the inequality for general subset $B\subset\square_R$, we let $m=m(R)\in\N$ be such that $3^{m-1}<R\le 3^m$.  Let $\tau=\min\{k\ge 0: X_k\notin B\}$ and $\tau^{(m)}=\min\{k\ge 0: X_k\notin Q_m\}$. Then, by  the strong Markov property,
\begin{align*}
\phi(x)&=- E^x_\omega\left[\sum_{i=0}^{\tau-1}\psi(\bar\omega^i)\right]
=\phi_m(x)-E_\omega^x[\phi_m(X_{\tau})]
\end{align*}
and so $\max_{B}|\phi|\le 2\max_{Q_m}|\phi_m|$. Using \eqref{eq:190130-2} and  $R_m/3<R\le R_m$, the proposition follows.
\end{proof}

\section{Proofs of quantitative RWRE results}\label{sec:quan-RWRE}

\subsection{Proof of Theorem~\ref{thm:quan-ergo}}\label{sec:pf-quan-ergo}
\begin{proof}[Proof of Theorem~\ref{thm:quan-ergo}]
First, we extend the definition of $T$ to be a stopping time for the space-time sequence $(X_m-X_0, m)_{m\in\N}$. 

Without loss of generality, we assume $T\le n$ almost surely and $E_\Q\psi=0$.
Let 
\[
R:=\sqrt n
\] 
and define a sequence of stopping times $\tau_0=0$,
\[
\tau_{k+1}=\min\left\{m>\tau_k: X_m-X_{\tau_k}\notin B_R\right\}, \quad \forall k\ge 0.
\]
Let $\tau(x,R):=\min\{m\ge 0: X_m\notin B_R(x)\}$. 
We say that a point $x\in\Z^d$ is ``good" if $\max_{y\in B_R(x)}|E_\omega^y[\sum_{i=0}^{\tau(x,R)-1}\psi(\bar\omega^i)]|<Cn^{1-0.5\alpha}$, where $\alpha$ is the same as in Proposition~\ref{prop:conv_rate_evfpvp}. Then, by Proposition~\ref{prop:conv_rate_evfpvp}, with $\mb P$-probability at least $1-Cn^de^{-cn^{p/2}}$, all points in the ball $B_{R^2}=B_n$ are good. 
Then, in  such a ball,
by the Markov property, 
\begin{align*}
&\Abs{E_\omega\bigg[\sum_{i=T}^{\tau_{k+1}-1}\psi(\bar\omega^i)\mathbbm{1}_{\tau_k<T\le\tau_{k+1}}\bigg]}\\
&=\sum_{y,m,z}\Abs{E_\omega\bigg[\mathbbm{1}_{\{X_{\tau_k}=y,T=m,\tau_k<m\le\tau_{k+1}, X_m=z\}}E_\omega^z\big[\sum_{i=0}^{\tau(y,R)-1}\psi(\bar\omega^i)\big]\bigg]}\\
&\le C\sum_{y,m,z}P_\omega\big(X_{\tau_k}=y,T=m,\tau_k<m\le\tau_{k+1}, X_m=z\big)n^{1-0.5\alpha}\\
&=Cn^{1-0.5\alpha}P_\omega(\tau_k<T\le\tau_{k+1})
\end{align*}
for $k\ge 0$. 
Similarly, in such a ball, we have for $k\ge 0$,
\begin{align*}
\Abs{E_\omega\left[\sum_{i=\tau_k}^{\tau_{k+1}-1}\psi(\bar\omega^i)\mathbbm{1}_{T>\tau_k}\right]}
\le 
Cn^{1-0.5\alpha}P_\omega(\tau_k<T).
\end{align*}
Hence, with $\mb P$-probability at least $1-Cn^de^{-cn^{p/2}}$, by the
two inequalities above,
\begin{align}\label{eq:181104-1}
\Abs{E_\omega\left[\sum_{i=0}^{T-1}\psi(\bar\omega^i)\right]}
&=
\Abs{\sum_{k=0}^{\infty}E_\omega\left[\sum_{i=\tau_k}^{\tau_{k+1}-1}\psi(\bar\omega^i)\mathbbm{1}_{T>\tau_k}\right]-E_\omega\left[\sum_{i=T}^{\tau_{k+1}-1}\psi(\bar\omega^i)\mathbbm{1}_{\tau_k<T\le \tau_{k+1}}\right]}\nn\\
&\le 
C\sum_{k=1}^\infty P_\omega(\tau_k<T)n^{1-0.5\alpha}.
\end{align}
By \cite[Lemma~4]{GZ-12}, there exist constants $c_1, c_2>0$ such that $E_\omega[e^{-c_1\tau(0,R)/R^2}]<e^{-c_2}$ for $\mb P$-almost every $\omega\in\Omega$. By the Markov property, $E_\omega[e^{-c_1\tau_k/R^2}]<e^{-kc_2}$ for all $k\ge 1$. Thus by Chebyshev's inequality,
\[
P_\omega(\tau_k<T)
\le 
P_\omega(\tau_k<n)\le E_\omega[e^{(n-\tau_k)c_1/R^2}]\le Ce^{-kc_2}.
\]
This inequality, together with \eqref{eq:181104-1}, yields the theorem.
\end{proof}

\subsection{Berry-Esseen estimate (Theorem~\ref{thm:BE})}
\begin{proof}
[Proof of Theorem~\ref{thm:BE}]
Let $\psi$ be any bounded
 local function
 with $E_\Q\psi=0$.
By Theorem~\ref{thm:quan-ergo}, with probability at least $1-Ce^{-cn^{p/2}}$, we have $|\qe^x[\sum_{i=0}^{m}\psi(\bar\omega^i)]|\le Cn^{1-\alpha}$ for all $x\in B_n$ and $0\le m\le n$.
Hence, with $\mb P$-probability at least $1-Ce^{-cn^{p/2}}$,
\begin{align*}
\qe\left[\left(\sum_{k=0}^{n-1}\psi(\bar\omega^k)\right)^2\right]
&\le 
2\sum_{i=0}^{n-1}E_\omega\left[\psi(\bar\omega^i)\sum_{j=0}^{n-i-1}\psi(\bar\omega^{i+j})\right]\\
&=
2\sum_{i=0}^{n-1}E_\omega\left[\psi(\bar\omega^i)\qe^{X_i}\left[\sum_{j=0}^{n-i-1}\psi(\bar\omega^{j})\right]\right]\\
&\le 
C\sum_{i=0}^{n-1} n^{1-\alpha}=Cn^{2-\alpha}
\end{align*}
and so
\[
\qe\left[\Abs{\tfrac 1n\sum_{k=0}^{n-1}\psi(\bar\omega^k)}^2\right]
\le 
Cn^{-\alpha}.
\]
In particular, for any unit vector $\ell\in\R^d$, applying the above inequality to $\psi(\omega)=\ell^T\tfrac{\omega(0)}{\tr\omega(0)}\ell$, we have,  with $\mb P$-probability at least $1-Ce^{-cn^{p/2}}$,
\begin{equation}\label{eq:181106-1}
\qe\left[
\Abs{
\frac{1}{n}\sum_{k=0}^{n-1}E_\omega\big[\big((X_{k+1}-X_k)\cdot\ell\big)^2|\ms F_k\big]
-\ell^T\bar a\ell
}^{2}
\right]
\le Cn^{-\alpha}
\end{equation}
where $\ms F_k:=\sigma(X_0,\ldots,X_k)$.
Set $S_n^2=E_\omega[(X_n\cdot\ell)^2]=E_\omega[\sum_{i=0}^{n-1}\ell^T\tfrac{\bar\omega^i(0)}{\tr\bar\omega^i(0)}\ell]$. 
By Theorem~\ref{thm:quan-ergo},   with $\mb P$-probability at least $1-Ce^{-cn^{p/2}}$, 
$|\tfrac 1n S_n^2-\ell^T\bar a\ell|\le n^{-\alpha}$, and so \eqref{eq:181106-1} yields
\[
\qe\left[
\Abs{
\frac{1}{S_n^2}\sum_{k=0}^{n-1}E_\omega\big[\big((X_{k+1}-X_k)\cdot\ell\big)^2|\ms F_k\big]
-1
}^{2}
\right]
\le Cn^{-\alpha}.
\]
Therefore, by 
a quantitative martingale CLT (\cite[Theorem~2]{EH-88} or \cite[Theorem~1.1]{JM-13}),
\[
\sup_{r\in\R}
\Abs{
P_\omega
\left(
X_{n}\cdot\ell/\sqrt n \le r\sqrt{\ell^T\bar a\ell}
\right)
-\Phi(r)
}
\le 
C(n^{-\alpha}+n^{-1})^{1/5}
\le 
Cn^{-(\alpha\wedge 1)/5}
\]
with $\mb P$-probability at least $1-Ce^{-n^{p/2}}$.
\end{proof}

\section{Quantitative bounds for the homogenization errors}\label{sec:quan-ellip-para}
In this section we will use quantitative versions of Berger's argument \cite{B-per} to bound the homogenization errors of both elliptic and parabolic non-divergence form difference operators.  
 The idea, which compares the diffusivity of the RWRE in large scale to that of the Brownian motion, is as explained at the end of Section~\ref{sec:intro}. 
Recall the notations in \eqref{def:omegaxei} and \eqref{def:omegabar}. Note that the covariance matrix $M^{(n)}=E_\omega[X_nX_n^T]/n$ of the large scale random walk has diagonal entries $M^{(n)}_{ii}=E_\omega[\sum_{k=0}^{n-1}\omega(X_k,e_i)]/n$, $i=1,\ldots,d$.  
Hence, with the quantification (Theorem~\ref{thm:quan-ergo}) of the ergodicity of the environment viewed from the particle $(\bar\omega_k)_{k\in\N}$, we can quantify the homogenization error by considering RWRE with long jumps (cf. definitions of stopping time $\sigma$ in both Subsections \ref{subsec:homo-ellip} and \ref{subsec:homo-para}).
Similar to the proof of Lemma~\ref{lem:conv_rate_evfpvp}, such quantifications only hold for ``good" environments. We will choose the jump-size of the large scale RWRE appropriately so that environments around a sufficiently big proportion of points in $B_R$ are good.
\subsection{The elliptic case: proof of Theorem~\ref{thm:homog-nondiv}}\label{subsec:homo-ellip}
Let $u, \bar u$ be as in Theorem~\ref{thm:homog-nondiv}. 
For $q\in(0,d)$, let $\gamma=\gamma(q)\in(0,0.5)$ be a constant whose value will be determined in  the last step of the proof of Theorem~\ref{thm:homog-nondiv}. Let $R_0:=R^\gamma$ and denote the exit time of the random walk from a ball (centered at the starting point) of radius $R_0$ as
\[
\sigma=\sigma(R_0):=\min\left\{n\ge 0: X_n-X_0\notin B_{R^\gamma}\right\}.
\]
\begin{definition}\label{def:good}
 Let $\alpha=\alpha(d,\kappa,1)>0$ and $C$ be as in Proposition~\ref{prop:conv_rate_evfpvp}.
 Let $\psi$ be a local function as in (A3).
 We say that a point $x$ is good (and otherwise bad) if for all $\zeta(\omega)\in\left\{\tfrac{\psi(\omega)}{\tr\omega(0)},\tfrac{\omega(0)}{\tr\omega(0)}\right\}$,
 \[
 \Abs{E_\omega^x\big[\sum_{i=0}^{\sigma-1}(\zeta(\bar\omega_i)-E_\Q\zeta)\big]}\le C\norm{\zeta}_\infty R_0^{2-\alpha}. 
\]
Note that by Proposition~\ref{prop:conv_rate_evfpvp}, $\mb P(x\text{ is bad})\le Ce^{-cR_0}$. 
 Moreover,  by (A1) and (A3), the event $\{x\text{ is bad}\}$ is independent of the environment $\{\omega(y):y\notin B_{R_0+2\dpd}\}$ under $\mb P$.
\item 
\end{definition}

\begin{proof}
[Proof of Theorem~\ref{thm:homog-nondiv}]

Since $g\in C^3(\partial\B_1)$, it can be extended to be a function in $C^{2,1}(\B_1)$ with $|g|_{2,1;\B_1}\le C|g|_{2,1;\partial\B_1}$. By \cite[Theorem 6.6]{GiTr} and ABP inequality, $|\bar u|_{2,1;\B_1}\le C (|f|_{0,1}+|g|_{2,1})$. 
Set 
\[
\bar u_R(x):=\bar u(\tfrac x R).
\]
Then, for $x\in B_R$,
\begin{align}\label{eq:181101}
L_\omega\bar u_{R+1}(x)
&=\tfrac{1}{2\tr\omega(x)}\sum_{i=1}^d\omega_i(x)[\bar{u}(\tfrac{x+e_i}{R+1})+\bar{u}(\tfrac{x-e_i}{R+1})-2\bar{u}(\tfrac{x}{R+1})]\nn\\
&=\tfrac{1}{2\tr\omega(x)R^2}\tr\left(\omega(x)D^2\bar{u}(\tfrac{x}{R})\right)+O(R^{-3}),
\end{align}
where $|O(R^{-3})|\le C (|f|_{0,1;\B_1}+|g|_{2,1;\B_1})R^{-3}$.  

Our proof of the theorem consists of a few steps, where the first two steps are 
to justify that the discrepancy between the discrete and continuous boundaries does not generate much error.
 In Steps 3 and 4 we control the homogenization error by comparing (the covariance matrices of) a large scale random walk to the Brownian motion at good points. In the last two steps, we obtain an exponential tail for the number of bad points.

\noindent{\bf Step 1.} We claim that in $B_R$, $\bar u_R(x)$ is very close to the solution $\hat{u}: \bar B_R\to\R$ of  
\[
\left\{
\begin{array}{lr}
L_\omega\hat{u}=L_\omega\bar u_{R+1} \quad&\text{ in }B_{R},\\
\hat{u}(x)=g(\tfrac{x}{R+1}) \quad&\text{ on }\partial B_{R}.
\end{array}
\right.
\]
To this end, let $u_+, u_-$ denote the functions $u_{\pm}(x)=g(\tfrac x{R+1})\pm C\frac{(R+1)^2-|x|^2}{R^2}$, where $C$ is a constant to be determined. Then $u_{\pm}=\bar u_{R+1}=g(\frac{x}{R+1})$ on $\partial\B_{R+1}$. Taking $C=C(|g|_{2,1})>0$ large enough, for $x\in\B_{R+1}$, we have $\tr [\bar{a}D^2 (u_+-\bar{u}_{R+1})]\le \tfrac{c}{R^2}([g]_{2,1;\B_1}+|f|_{0,1;\B_1}-C)\le 0$, and similarly $\tr [\bar{a}D^2 (u_--\bar{u}_{R+1})]\ge 0$. Hence by the comparison principle,
\[
u_-\le \bar u_{R+1}\le u_+ \quad \text{ in }\B_{R+1}.
\]
In particular, for $x\in\partial B_R$, 
\[
|\bar u_{R+1}(x)-g(\tfrac x{R+1})|
\le 
C\tfrac{(R+1)^2-|x|^2}{R^2}
\le 
\tfrac{C}{R}.
\]
Thus,
noting that $L_\omega(\hat u-\bar u_{R+1})=0$ in $B_R$, by \eqref{eq:u-mu-max}  and Lemma~\ref{lem:subdiff-ub} we get
\[
\max_{B_R}|\hat u-\bar u_{R+1}|\le \tfrac{C}{R},
\]
which, together with the Lipschitz continuity of $\bar u$, yields
\[
\max_{B_R}|\hat u-\bar u_R|\le \tfrac CR.
\]

\noindent{\bf Step 2.} 
Now let $\tilde u$ be the solution of 
\[
\left\{
\begin{array}{lr}
L_\omega\tilde{u}=\tfrac{1}{2\tr\omega(x)R^2}\tr\big(\omega(x) D^2\bar{u}(\tfrac{x}{R})\big) &\text{ in }B_R,\\
\tilde{u}=g(\tfrac{x}{|x|}) &\text{ on }\partial B_R. 
\end{array}
\right.
\]
Then by \eqref{eq:181101} and the Lipschitz continuity of $f$ and $g$, $|L_\omega(\tilde u-\hat u)|\le C/R^3$ in $B_R$ and $|\tilde{u}-\hat u|\le C/R$ 
on $\partial B_R$. By Lemma~\ref{lem:abp} and Lemma~\ref{lem:subdiff-ub}, we get
$\max_{B_R}|\tilde{u}-\hat{u}|\le C/R$ and so by the previous step,
\[
\max_{B_R}|\tilde{u}-\bar u_R|\le \tfrac CR.
\]

\noindent{\bf Step 3.} 
It remains to bound $\max_{B_R}|u-\tilde{u}|$. 
Let $v:=\tilde{u}-u$.  We will define a small perturbation $w$ of $v$, which has a small subdifferential set (see \eqref{eq: emptysubdiff} below).
Notice that 
\[
\left\{
\begin{array}{lr}
L_\omega v=\tfrac{-1}{2R^2}\tr\left(\left(\bar a - \tfrac{\omega(x)}{\tr \omega(x)}\right)D^2\bar{u}(\tfrac{x}{R})\right)+\frac{1}{R^2}f(\tfrac xR)\left(\bar \psi-\tfrac{\psi_\omega(x)}{\tr\omega(x)}\right)
 &\text{ in }B_R,\\
v|_{\partial B_R}=0 &\text{ on }\partial B_R.
\end{array}
\right.
\]
Set $\psi_0(\omega):=\tfrac{\psi_\omega(x)}{\tr(\omega)} -\bar \psi$, and $\omega_0= \tfrac{\omega}{\tr(\omega)}$. 
Then
\begin{align*}
E_\omega^x[v(X_{\sigma})-v(x)]
&=
\tfrac{-1}{R^2}E_\omega^x\left[\sum_{i=0}^{\sigma-1}\tfrac12\tr[(\bar a-\bar\omega_0^i)D^2\bar u(\tfrac{X_i}{R})]+f(\tfrac{X_i}R)\psi_0(\bar\omega^i)\right].
\end{align*}
Since $\bar a, \bar\omega_0$ are bounded matrices, 
$|D^2\bar u(\tfrac{y}{R})-D^2\bar u(\tfrac{x}{R})|\le CR_0/R=CR^{\gamma-1}$ 
and similarly $|f(\tfrac yR)-f(\tfrac xR)|\le CR^{\gamma-1}$
for any $y\in B_{R_0}(x)$,
we have for any good point $x\in B_{R-R_0}$,
\begin{align*}
&E_\omega^x[v(X_{\sigma})-v(x)]\\
&=
-\tfrac{1}{R^2}E_\omega^x\left[\sum_{i=0}^{\sigma-1}\tfrac12\tr[(\bar a-\bar\omega_0^i)D^2\bar u(\tfrac{x}{R})]+f(\tfrac{x}R)\psi_0(\bar\omega^i)\right]+O(R^{\gamma-1})\frac{E_\omega^x[\sigma]}{R^2}\\
&=O(R^{(2-\alpha)\gamma-2}+R^{3(\gamma-1)})=O(R^{(2-\alpha)\gamma-2}).
\end{align*}
We let $\tau_R=\min\{n\ge 0:X_n\notin B_R\}$ and set 
\[
w(x)=v(x)+C_1R^{-\alpha\gamma-2}E^x_\omega[\tau_R], 
\]
 where $C_1>0$ is a constant to be determined. 
Then, for any good point $x\in B_{R-R_0}$,
\[
E_\omega^x[w(X_{\sigma})-w(x)]
=O(R^{(2-\alpha)\gamma-2})-C_1R^{-\alpha\gamma-2}E_\omega^x[\sigma]<0
\]
if $C_1$ is chosen to be large enough since $E^x_\omega[\sigma]\ge R_0^2$.
This implies
\begin{equation}\label{eq: emptysubdiff}
\partial w(x;B_R)=\emptyset \quad \text{for any good point }x\in B_{R-R_0},
\end{equation}
because otherwise there exists $p\in\R^d$ such that 
$0\le E_\omega^x[w(X_{\sigma})-w(x)-p\cdot(X_{\sigma}-x)]=E_\omega^x[w(X_{\sigma})-w(x)]$.
Here, we used the optional stopping theorem and the fact that $X_n$ is a martingale.

\noindent{\bf Step 4.} 
Now, we will apply the ABP inequality to bound $|v|$ from the above. 
Since
 $L_\omega w=L_\omega\tilde{u}-L_\omega u-C_1R^{-\alpha\gamma-2}\le C/R^2$,  by Lemma~\ref{lem:subdiff-ub}, $|\partial w(x;B_R)|\le CR^{-2d}$ for $x\in B_R$. Let
\[
 \ms B_R=\ms B_R(\omega,\gamma):=\#\text{bad points in }B_{R-R_0}.
\]
Display \eqref{eq: emptysubdiff} then yields
\begin{align*}
|\partial w(B_R)|
&\le [\ms B_R+\#(B_R\setminus B_{R-R_0})]CR^{-2d}\\
&\le C(\ms B_R+R^{d+\gamma-1})R^{-2d}.
\end{align*}

By Lemma~\ref{lem:abp}, $\min_{B_R}w\ge -CR|\partial w(B_R)|^{1/d}
\ge -C\ms B_R^{1/d}R^{-1}-CR^{(\gamma-1)/d}$. Therefore, noting that $E^x_\omega[\tau_R]\le (R+1)^2$ and choosing $\alpha<1/d$,
\begin{align*}
\min_{B_R}(\tilde u-u)
&\ge \min_{B_R}w-C_1R^{-\alpha\gamma-2}\max_{x\in B_R}E_\omega^x[\tau_R]\\
&\ge 
-C(\ms B_R^{1/d}R^{-1}+R^{-(1-\gamma)/d}+R^{-\alpha\gamma})\\
&\ge -C(\ms B_R^{1/d}R^{-1}+R^{-\alpha\gamma}).
\end{align*}
Similar upper bound for $\max_{B_R}(\tilde u-u)$ can be obtained by substituting $f,g$ by $-f,-g$ in the problem.
This, together with Step~2, yields
\[
\max_{B_R}|\bar u_R-u|\le 
C(\ms B_R^{1/d}R^{-1}+R^{-\alpha\gamma}).
\]

\noindent{\bf Step 5.} 
 Without loss of generality we only consider $R$ sufficiently big such that $R_0>\dpd$. 
We will show that
 \begin{equation}\label{eq:bad-pt-moment}
\mb E[\exp(cR^{\gamma(1-d)}\ms B_R)]<C.
\end{equation}
To see this, observe that we can cover the ball $B_{R-R_0}$ with $(6R_0)^d$ (not necessarily disjoint) subsets $S_i, i\in I:=\{1,\ldots, (6R_0)^d\}$, such that for each $i\in I$, $\#S_i\le C(R/R_0)^d$ and $\dist(x,y)>6R_0 >2R_0+4\dpd$ for any $x,y\in S_i$. 
In other words, $\{\mathbbm{1}_{x\text{ is bad}}: x\in S_i\}$ are independent random variables.
Since for $x\in\Z^d$, $c<C/2$,
\begin{align*}
\mb E[\exp(cR_0\mathbbm{1}_{x\text{ is bad}})]
&\le 
e^{cR_0}\mb P(x\text{ is bad})+1\\
&\le e^{cR_0}e^{-CR_0}+1
\le 
1+e^{-cR_0}, 
\end{align*}
we have, for each $\ms B^i:=\#\text{bad points in }S_i$,
\[
\mb E[\exp(cR_0\ms B^i)]\le (1+e^{-cR_0})^{C(R/R_0)^d}\le C.
\]
Hence, using H\"older's inequality, 
\begin{align*}
\mb E[\exp(cR_0^{1-d}\ms B_R)]
&\le \mb E[\prod_{i\in I}\exp(cR_0^{1-d}\ms B^i)]\\
&\le 
\prod_{i\in I}\norm{\exp(cR_0^{1-d}\ms B^i)}_{L^{R_0^d}(\mb P)}\\
&=\prod_{i\in I}(\mb E[\exp(cR_0\ms B^i)])^{1/R^d_0}<C.
\end{align*}

\noindent{\bf Step 6.} 
Let $\ms X_R=R^{-\gamma}\ms B_R^{1/d}$ and
\[
\ms X=\max_{R\ge 1}\ms X_R.
\]
Then
 by Chebyshev's inequality, for $t\ge 1$, $R\ge 1$,
\begin{align*}
\mb P(\ms X_R>t)
&\le 
\mb E[\exp(cR^{-\gamma(d-1)}\ms B_R-cR^{\gamma}t^d)]
\\&
\le 
C\exp(-cR^{\gamma}t^d)\le 
C_\gamma R^{-2}\exp(-ct^d),
\end{align*}
where $C_\gamma$ depends on $(\gamma,d,\kappa, \dpd)$. 
Thus by a union bound, for $t\ge 0$, $\mb P(\ms X>t)\le C_\gamma e^{-ct^d}$ and
\[
\mb E[\exp(c\ms X^d)]<\infty.
\]

By Step~4 and the definition of $\ms X$, we have $\max_{B_R}|\bar u_R-u|\le C(\ms X R^{\gamma(\alpha+1)-1}+1)R^{-\alpha\gamma}$. 
The theorem follows by taking $\gamma\le (d-q)/[d(1+\alpha)]$.
\end{proof}

\subsection{The parabolic case: proof of Theorem~\ref{thm:homog-parab}}\label{subsec:homo-para}
The proof of the parabolic case also uses 
a quantification (Theorem~\ref{thm:quan-ergo}) of the ergodicity of the environment from the point of view of the particle
and follows similar ideas as the elliptic case. Note that unlike elliptic operators, linear parabolic operators are related to the stochastic processes $\hat Y_n$ on $\Z^d\times\Z$ defined below.

Let $\hat Y_n=(Y_n,T_n)$ be a Markov chain on $\Z^d\times\Z$ with transition probability
\[
P_\omega\left(\hat Y_{n+1}=(y,m+1)|\hat Y_n=(x,m)\right)=
\left\{
\begin{array}{lr}
\omega_i(x)/[2(1+\tr\omega)] &\text{ if }y=x\pm e_i,\\
1/(1+\tr\omega) &\text{ if }y=x,\\
0 &\text{ otherwise.}
\end{array}
\right.
\]
Note that the time coordinate $T_n=T_0+n$ of $\hat Y_n$ grows linearly. 
Denote the law of $\hat Y_n$ with initial state $\hat Y_0=\hat x$ by $P^{\hat x}_\omega$ and let $E^{\hat x}_\omega$ be its expectation.
For a function $u: \Z^d\times\Z\to\R$,
the corresponding parabolic operator for the process $\hat Y_n$ is
\begin{align*}
\mathscr{L}_\omega u(x,n)
:=\left(\tfrac 12\tr(\omega\nabla^2u(x,n+1))+[u(x,n+1)-u(x,n)]\right)/(1+\tr\omega).
\end{align*}
Clearly, $\ms L_\omega u(\hat x)=E^{\hat x}_\omega[u(\hat Y_1)]-u(\hat x)$. 

\begin{remark}
We have the following comments.
\begin{enumerate}
\item A main difference between $(Y_n)$ and  the random walk $(X_n)$ defined in \eqref{eq:def-RW} is that $(Y_n)$ has positive probability to stay put. 
In particular, $(Y_n)$ can be considered as a time changed process of $(X_n)$. 
\item Denote the environment viewed from the point of the particle $(Y_n)$ as
\begin{equation}\label{def:evfpofop-para}
\hat\omega^i=\theta_{Y_i}\omega\in\Omega, i\ge 0.
\end{equation}
By Theorem~\ref{thm:QCLT}, the Markov chain $\hat\omega_i$ has an invariant ergodic measure $\hat\Q$ that is mutually absolutely continuous with respect to $\mb P$. It can be checked that 
\begin{equation}\label{eq:RN-der}
\hat\Q(\dd\omega)=\frac{(1+\tr\omega)/\tr\omega}{E_\Q[(1+\tr\omega)/\tr\omega]}\Q(\dd\omega). 
\end{equation}
\item Theorem~\ref{thm:quan-ergo} also holds for balanced random walks with stay-put. Indeed, 
for any bounded local function $\psi$
with $E_{\hat \Q}[\psi]=0$, let $\tau_R=\min\{k\ge 0: Y_k\notin B_R\}$. Then, under the notation of \eqref{eq:def-of-L}, $u(x)=E_\omega^x[\sum_{i=0}^{\tau_R-1}\psi(\theta_{Y_i}\omega)]$ solves the Dirichlet problem with $u|_{\partial B_R}=0$ and $L_\omega u=\tfrac{1+\tr\omega}{\tr\omega}\psi$ in $B_R$. 
Note that $E_\Q[\tfrac{1+\tr\omega}{\tr\omega}\psi ]=0$ by \eqref{eq:RN-der}. By Proposition~\ref{prop:conv_rate_evfpvp} (Let $\alpha,p$ be the same as therein.)
\[
\mb P\left(\max_{x\in B_R}|u(x)|\ge CR^{-\alpha}\right)\le Ce^{-cR^p}.
\]
Then, exactly the same argument as in the proof of Theorem~\ref{thm:quan-ergo} (Section~\ref{sec:pf-quan-ergo}) shows that for any $p\in(0,d)$, there exists $\alpha=\alpha(d,\kappa,\dpd,p)>0$ such that
\begin{equation}\label{eq:quant-stayput}
 \mb P\left(\Abs{\frac 1n
 E_\omega\big[\sum_{i=0}^{T\wedge n-1}(\psi(\theta_{Y_i}\omega)-E_{\hat\Q}\psi)\big]}\ge C\norm{\psi}_\infty n^{-\alpha/2}\right)\le Ce^{-cn^{p/2}}
\end{equation}
for any stopping time $T$ of the random walk $(Y_i)$.
\end{enumerate}
\end{remark}

Similar to Section~\ref{subsec:homo-ellip}, we use a discrete parabolic ABP estimate to control solutions of the Dirichlet problem \eqref{eq:para-dirich}.  For any function $u:\bar K_R\to\R$ and $(x,n)\in K_R$, define the parabolic subdifferential sets
\begin{align*}
&\partial u(x,n)=\partial u(x,n;K_R)\\
&=\left\{p\in\R^d: u(y,m)-u(x,n)\ge p\cdot(y-x) \text{ for all }(y,m)\in K_R\cup\partial^p K_R
\text{ with }m>n\right\}.
\end{align*}
and let
\[
\ms Du(x,n)=\left\{(p,q-p\cdot x): p\in\partial u(x,n), q\in[u(x,n),u(x,n+1)]\right\}\subset\R^{d+1}.
\]
The following discrete parabolic ABP inequality is implicitly contained in the proof of \cite[Theorem 2.2]{DGR-15}. 
For the  purpose of completeness, we include its proof in the appendix.
\begin{theorem}
[Parabolic ABP inequality]\label{thm:para-abp}
There exists a constant $C=C(d)$ such that for any function $u:\bar K_R\to\R$,
\[
\min_{\partial^p K_R}u
\le 
\min_{K_R}u+CR^{\tfrac d{d+1}}\left(
\sum_{(x,n)\in K_R}[u(x,n+1)-u(x,n)]|\partial u(x,n)|
\right)^{1/(d+1)}.
\]
\end{theorem}
Note that for any $(x,n)\in K_R$, 
\begin{equation}\label{eq:181127}
|\ms D u(x,n)|= [u(x,n+1)-u(x,n)]|\partial u(x,n;K_R)|.
\end{equation}
Similar to Lemma~\ref{lem:subdiff-ub}, we have an upper bound for $|\ms D u(x,n)|$ in terms of $\ms L_\omega$.
\begin{lemma}
\label{lem:para-subdiff-ub}
Assume that $\omega\in\Omega$ and $\kappa I\le \omega(x)\le \kappa^{-1}I$ for all $x$.
There exists $C=C(d,\kappa)$ such that for $u:\bar{K}_R\to\R$ with $u|_{\partial^p K_R}=0$ and any $\hat x=(x,n)\in K_R$
\[
|\ms D u(\hat x)|\le C\left(\ms L_\omega u\right)_+^{d+1}.
\]
\end{lemma}
\begin{proof}
By the same argument as in Lemma~\ref{lem:subdiff-ub}, $|\partial u(\hat x;K_R)|\le (\ms L_\omega u)_+^d$. Moreover, when $\ms Du(x,n)\neq\emptyset$, 
then $u(x,n+1)-u(x,n)\le C\ms L_\omega u(\hat x)$ by uniform-ellipticity. 
\end{proof}

For $\hat x=(x,n)$, set 
\[
\hat x^{(R)}=(\tfrac xR, \tfrac n{R^2}), \qquad\bar u_R(\hat x):=\bar u(\hat x^{(R)}).
\]
For simplicity of notation, we set
\[
\psi_0(x)=\tfrac{\psi(\theta_x\omega)}{1+\tr\omega(x)}, \quad \xi=\tfrac{\omega}{1+\tr\omega}, \quad b_0=\tfrac{1}{1+\tr\omega}, 
\] 
and write the $\hat\Q$-expectations of $\psi_0,\xi,b_0$ as $\bar\psi_0,\bar\xi,\bar b_0$, respectively. Denote these quantities viewed from the viewed of $Y_n$ as
\[
\hat\psi_0^i:=\psi_0(Y_i), \quad \hat\xi^i:=\xi(Y_i),\quad \hat b^i_0:=b_0(Y_i).
\]

We define good points similarly as in the elliptic case.
Let $\gamma\in(0,1/3)$ be a constant whose value will be determined at the end of the proof of Theorem~\ref{thm:homog-parab}. Set $R_0:=R^\gamma$ and 
\[
\sigma=\sigma(R_0)=:
\min\left\{k\ge 0: \hat Y_k-\hat Y_0\notin K_{R_0}\right\}.
\] 
\begin{definition}
Let $\alpha=\alpha(d,\kappa,\dpd, 1)$ and $C$ be the same as in \eqref{eq:quant-stayput}. We say that a point $x\in\Z^d$ is good (and otherwise bad) if for all $\zeta\in\{\psi_0,\xi,b_0\}$, 
\[
\Abs{
E_\omega^x\big[\sum_{i=0}^{\sigma-1}(\zeta(Y_i)-E_{\hat\Q}[\zeta(0)])\big]
}
\le C(1+\norm{\psi}_\infty)R_0^{2-\alpha}.
\]
Note that by \eqref{eq:quant-stayput}, $\mb P(x \text{ is bad})\le Ce^{-cR_0}$.
\end{definition}

Recalling \eqref{eq:RN-der}, both \eqref{eq:para-dirich} and its effective equation \eqref{eq:effective-para} can be rewritten as 
\begin{equation}
\label{eq:para-dirch-n}
\left\{
\begin{array}{lr}
\ms L_\omega u(\hat x)
=\tfrac{1}{R^2}f(\hat x^{(R)})\psi_0(\theta_x\omega) &\hat x\in K_R,\\
u(x,n)=g\left(\tfrac{x}{|x|\vee\sqrt n}, \tfrac{n}{|x|^2\vee n} \right)  &(x,n)\in\partial^p K_R,
\end{array}
\right.
\end{equation}
and
\[
\left\{
\begin{array}{lr}
\tfrac12\tr(\bar\xi D^2\bar u)+\bar b_0\partial_t\bar u=f\bar\psi_0\quad&\text{ in }\K_1,\\
\bar u=g \quad&\text{ on }\partial^p\K_1.
\end{array}
\right.
\]
For $\hat x=(x,n)\in K_R$, by the Lipschitz continuity of $D^2\bar u$ and $\partial_t \bar u$, 
\begin{align}\label{eq:181130}
&\ms L_\omega\bar u_{R+1}(\hat x)\\
&=\tfrac 12\tr(\xi(x)\nabla^2\bar u_{R+1}(x,n+1))+b_0(x)[\bar u_{R+1}(x,n+1)-\bar u_{R+1}(x,n)]\nn\\
&=\tfrac 1{R^2}\left[\tfrac 12\tr\left(\xi D^2\bar u(\tfrac xR, \tfrac{n+1}{(R+1)^2})\right)+b_0(x)\partial_t\bar u(\hat x^{(R)})
\right]+O(R^{-3})\nn\\
&=\tfrac 1{R^2}f(\hat x^{(R)})\bar\psi_0+\tfrac1{2R^2}\tr\left((\xi-\bar\xi)D^2\bar u(\hat x^{(R)})\right)+\tfrac{1}{R^2}(b_0-\bar b_0)\partial_t\bar u(\hat x^{(R)})+O(R^{-3}).\nn
\end{align}	
\begin{proof}[Proof of Theorem~\ref{thm:homog-parab}:]
Since $g\in C^3_2(\partial^p \K_1)$, it can be extended to be a function in $C^3_2(\K_1)$ with $|g|_{C^3_2(\K_1)} \leq C|g|_{C^3_2(\partial^p \K_1)}$.

\noindent{\bf Step 1.}  Let $u_1: \bar K_R\to\R$ be the solution of 
\[
\left\{
\begin{array}{lr}
\ms L_\omega u_1=\ms L_\omega\bar u_{R+1} \qquad &\text{in }K_R,\\
u_1(\hat x)=g(\hat x^{(R+1)})\qquad &\text{on }\partial^p K_R.
\end{array}
\right.
\]
We claim that $\bar u_R$ is very close to $u_1$. Indeed, for $\hat x=(x,t)$, 
define functions $u_{\pm}(\hat x)=g(\hat x^{(R+1)})\pm C\frac{(R+1)^2-|x|^2}{R^2}$ , where $C$ is a constant to be determined. Then $u_{\pm}\gtrless\bar u_{R+1}$ on $\partial^p\K_{R+1}$. Moreover, $\ms L_\omega (u_{\pm}-\bar u_{R+1})\lessgtr\frac{c}{R^2}(\norm{g}_{C_2^3(\K_1)}\pm |f|_{0,1;\K_1}\mp C)\lessgtr 0$ in $K_{R+1}$ if $C$ is chosen to be large enough. Hence by the comparison principle, $u_-\le\bar u_{R+1}\le u_+$ in $K_{R+1}$. In particular, for $\hat x\in\partial^l K_R$ in the lateral boundary, 
\[
|\bar u_{R+1}(\hat x)-g(\hat x^{(R+1)})|\le C\tfrac{(R+1)^2-|x|^2}{R^2	}\le \tfrac{C}{R}.
\]
To obtain the same control in the time boundary $\partial^t K_R$, we let $v_{\pm}(\hat x)=g(\hat x^{(R+1)})\pm\frac{C}{R^2}[(R+1)^2-t]$. Similarly we have $\ms L_\omega(v_{\pm}-\bar u_{R+1})\lessgtr 0$ in $K_{R+1}$ if $C$ is large enough, and $v_{\pm}\gtrless \bar u_{R+1}$ in $\partial^p K_{R+1}$. Thus $v_-\le \bar u_{R+1}\le v_+$ in $K_{R+1}$ and so for $\hat x\in\partial^t K_R$, 
\[
|\bar u_{R+1}(\hat x)-g(\hat x^{(R+1)})|\le C\tfrac{(R+1)^2-t}{R^2}\le \tfrac{C}{R}.
\]
The two displays above and the comparison principle implies $\max_{K_R}|u_1-\bar u_{R+1}|\le C/R$, which, together with the regularity of $\bar u$, yields
\[
\max_{K_R}|u_1-\bar u_R|\le C/R.
\]

\noindent{\bf Step 2.}  Let $u_2$ be the function that satisfies $u_2=u$ in $\partial^p K_R$ and
\[
\ms L_\omega u_2=\tfrac 1{R^2}f(\hat x^{(R)})\bar\psi_0+\tfrac1{2R^2}\tr\left((\xi-\bar\xi)D^2\bar u(\hat x^{(R)})\right)+\tfrac{1}{R^2}(b_0-\bar b_0)\partial_t\bar u(\hat x^{(R)})\]
in $K_R$. 
By \eqref{eq:181130} and the Lipschitz continuity of $f$ and $D^2\bar u$, $|\ms L_\omega(u_1-u_2)|\le C/R^3$ in $K_R$. Also, $\max_{\partial^p K_R}|u_1-u_2|\le C/R$. By Theorem~\ref{thm:para-abp} and Lemma~\ref{lem:para-subdiff-ub} we get
$\max_{K_R}|u_1-u_2|\le C/R$ and so by the previous step
\[
\max_{K_R}|u_2-\bar u_R|\le C/R.
\]
\noindent{\bf Step 3.} 
 It remains to bound $\max_{K_R}|u_2-u|$. Let $v:=u_2-u$. Then $v$ satisfies $v|_{\partial^pK_R}=0$ and 
\[
\ms L_\omega v=\tfrac 1{R^2}\left[f(\hat x^{(R)})(\bar\psi_0-\psi_0)+\tfrac1{2}\tr\left((\xi-\bar\xi)D^2\bar u(\hat x^{(R)})\right)+(b_0-\bar b_0)\partial_t\bar u(\hat x^{(R)})\right]
\]
in $K_R$.
Thus for $\hat x=(x,n)\in B_{R-R_0}\times[0, R^2-R_0^2)$, 
\begin{align*}
&E_\omega^{\hat x}[v(\hat Y_{\sigma})-v(\hat x)]\\
&=\tfrac{1}{R^2}E_\omega^{\hat x}\big[\sum_{i=0}^{\sigma-1}f(\hat Y_i^{(R)})(\bar\psi_0-\hat\psi_0^i)+\tfrac1{2}\tr\big((\hat\xi^i-\bar\xi)D^2\bar u(\hat Y_i^{(R)})
\big)
\\&
+(\hat b_0^i-\bar b_0)\partial_t\bar u(\hat Y_i^{(R)})\big].
\end{align*}
Since $\hat\xi^i, \bar\xi$ are bounded matrices, 
$|D^2\bar u(\hat y^{(R)})-D^2\bar u(\hat x^{(R)})|\le CR_0/R=CR^{\gamma-1}$ 
and $|f(\hat y^{(R)})-f(\hat x^{(R)})|\le CR^{\gamma-1}$
for any $\hat y\in \hat x+B_{R_0}\times[0,R_0^2)$,
for any good point $x\in B_{R-R_0}$ and $n\in[0,R^2-R^2_0)$, we have
\begin{align*}
E_\omega^{\hat x}[v(\hat X_{\sigma})-v(\hat x)]
&=
\tfrac{1}{R^2}E_\omega^{\hat x}\big[\sum_{i=0}^{\sigma-1}f(\hat x^{(R)})(\bar\psi_0-\hat\psi_0^i)+\tfrac1{2}\tr\big((\hat\xi^i-\bar\xi)D^2\bar u(\hat x^{(R)})
\big)\\&
+(\hat b_0^i-\bar b_0)\partial_t\bar u(\hat x^{(R)})\big]+O(R^{\gamma-1})\frac{E_\omega^{\hat x}[\sigma]}{R^2}\\
&=O(R^{(2-\alpha)\gamma-2}+R^{3(\gamma-1)})=O(R^{(2-\alpha)\gamma-2}).
\end{align*}
Thus, setting 
$\tau_R=\min\{n\ge 0: \hat Y_n\notin K_R\}$ and
(with $C_1>0$ to be determined)
\[
w(\hat x)=v(\hat x)+C_1R^{-\alpha\gamma-2}E^{\hat x}_\omega[\tau_R], 
\]
for any good point $x\in B_{R-R_0}$ and $n\in[0,R^2-R_0^2)$, we get
\[
E_\omega^{\hat x}[w(\hat Y_{\sigma})-w(\hat x)]
=O(R^{(2-\alpha)\gamma-2})-C_1R^{-\alpha\gamma-2}E_\omega^{\hat x}[\sigma]<0
\]
by taking $C_1$ large enough. (Note $E^x_\omega[\sigma]\ge c R_0^2$.) This implies that
\begin{equation}\label{eq: emptysubdiff-para}
\ms D w(x,n;K_R)=\emptyset
\end{equation}
for any good point $x\in B_{R-R_0}$ and $n\in[0,R^2-R_0^2)$.

\noindent{\bf Step 4.} 
Now we will apply the parabolic ABP inequality to bound $|v|$ from the above. 
Since
 $\ms L_\omega w=\ms L_\omega u_2-\ms L_\omega u-C_1R^{-\alpha\gamma-2}\le C/R^2$,  by Lemma~\ref{lem:para-subdiff-ub}, $|\ms D w(\hat x;K_R)|\le CR^{-2(d+1)}$ for $\hat x\in K_R$. 
Let 
\[
\ms B_R=\ms B_R(R_0,R,\omega):=\#\text{bad points in }B_{R-R_0}.
\]
 Display \eqref{eq: emptysubdiff-para} then yields that 
\begin{align*}
|\ms D w(K_R)|
&\le [\ms B_R R^2+\#K_R\setminus (B_{R-R_0}\times[0,R^2-R_0^2))]CR^{-2(d+1)}\\
&\le C(\ms B_R+R_0R^{d-1})R^{-2d}.
\end{align*}

By Theorem~\ref{thm:para-abp}, 
\[
\min_{K_R}w\ge -CR^{d/(d+1)}|\ms D w(K_R)|^{1/(d+1)}
\ge 
-C(\tfrac{\ms B_R}{R^d})^{1/(d+1)}-CR^{-(1-\gamma)/(d+1)}.
\]
 Therefore, noting that $E^{\hat x}_\omega[\tau_R]\le R^2$ and choosing $\alpha<1/(d+1)$,
\begin{align*}
\min_{K_R}(u_2-u)
&\ge \min_{K_R}w-C_1R^{-\alpha\gamma-2}\max_{\hat x\in K_R}E_\omega^{\hat x}[\tau_R]\\
&\ge 
-C(\tfrac{\ms B_R}{R^d})^{1/(d+1)}-CR^{-(1-\gamma)/(d+1)}-CR^{-\alpha\gamma}\\
&\ge 
-C(\tfrac{\ms B_R}{R^d})^{1/(d+1)}-CR^{-\alpha\gamma}.
\end{align*}
Similarly, substituting $f,g$ by $-f,-g$ in the problem, we get
\[
\max_{K_R}(u_2-u)\le 
C(\tfrac{\ms B_R}{R^d})^{1/(d+1)}+CR^{-\alpha\gamma}.
\]
The above inequalities, together with Step~2, yields
\[
\max_{K_R}|\bar u_R-u|\le C(1+\ms B_R^{1/(d+1)}R^{\alpha\gamma-d/(d+1)})R^{-\alpha\gamma}.
\]
\noindent{\bf Step 5.} 
By the same argument as in the proof of \eqref{eq:bad-pt-moment} in the elliptic case, we also get $\mb E[\exp(cR^{\gamma(1-d)}\ms B_R)]<C$. Setting 
\[
\ms Y=\max_{R\ge 1}R^{-d\gamma/(d+1)}\ms B_R^{1/(d+1)}, 
\]
similar argument as in the previous subsection gives $\mb E[\exp(c\ms Y^{d+1})]<\infty$.

The theorem follows by choosing $\gamma=\gamma(q,\alpha)$ appropriately.
\end{proof}

\begin{remark}
Finally, we give some comments about i.i.d.\,structure and uniform ellipticity used in the paper as following.
\begin{enumerate}

\item Our arguments might be applicable to ergodic environments with appropriate mixing rates. 
Of course, in mixing cases, there may be different rates for the homogenization errors and different probability estimates (depending on how mixing the environment is). 
See also remarks in \cite{AS-14}. 
It would be  interesting to figure out the corresponding results in this direction.

\item We assume uniform ellipticity for technical reasons (e.g., the use of ABP and Harnack inequalities for uniform elliptic operators). As \cite{GZ-12, BD-14} show, in i.i.d.\,environment, the loss of ellipticity does not prevent the RWRE to be diffusive in large scale. 
We believe that obtaining quantitative homogenization results  in a balanced random environment without uniform ellipticity is an important open problem.
\end{enumerate}
\end{remark}

\section*{Appendix: Proof of the discrete parabolic ABP inequality}
The following proof of Theorem~\ref{thm:para-abp} is inspired by the results of Deuschel, Guo, Ramirez \cite{DGR-15}.
We include it here for the purpose of completeness.
\begin{proof}
[Proof of Theorem~\ref{thm:para-abp}]
Without loss of generality, assume 
$\min_{\partial^p K_R}u=0$, and for some $\hat x_0=(x_0,n_0)\in K_R$,
\[M:=-u(\hat x_0)=-\min_{\bar K_R}u>0.\]
Set $\Lambda=\left\{(\xi,h)\in\R^d\times\R:R|\xi|<h<\tfrac M2\right\}$. Then $|\Lambda|=CM^{d+1}/R^d$. To prove the theorem, it suffices to show that 
\[
\Lambda\subset\ms  D u(K_R):=\bigcup_{\hat x\in K_R}\ms Du(\hat x;K_R).\]
For any fixed $(\xi,h)\in\Lambda$,  set $\phi(x,n)=u(x,n)+\xi\cdot x+h$. By the definition of $\Lambda$, we have $\phi(\hat x_0)<0$. We claim that there exists $\hat x_1=(x_1,n_1)\in K_R$ with $n_1\ge n_0$ such that $\phi(\hat x_1)\le 0$ and $(\xi,h)\in\ms D u(\hat x_1;K_R)$. Indeed, for $x\in B_R$, let 
\[
N_x=\max\left\{n: (x,n)\in K_R\text{ and }\phi(x,n)\le 0\right\}.
\]
Here we use the convention $\max\emptyset=-\infty$. We define $(x_1,n_1)$ to be such that
\[
n_1:=N_{x_1}=\max_{x\in B_R} N_x\ge N_{x_0}\ge n_0.
\]
Then, for any $\hat x=(x,n)\in K_R$ with $n>n_1$, we have $\phi(\hat x)>0\ge \phi(\hat x_1)$. Also, for any $\hat x\in\partial^p K_R$, by the definition of $\Lambda$ and $\phi$, we have $\phi(\hat x)>0$ and so $\xi\in\partial u(\hat x_1;K_R)$. Moreover, $u(x_1,n_1)\le h-x_1\cdot\xi<u(x_1,m)$ for any $m>n_1$. Therefore $(\xi,h)\in\ms D u(\hat x_1)$. The theorem follows by using \eqref{eq:181127}.
\end{proof}




\begin{acks}
We thank two anonymous referees whose comments improve the presentation of our article. 
We thank Scott Armstrong and Charlie Smart for letting us know their new version \cite{AS-20} of \cite{AS-14} on arXiv.  XG did the  main part of his work while at the University of Wisconsin Madison. He thanks Professors Timo Sep\"{a}l\"{a}inen  and other colleagues at the UW for their hospitality and the supportive environment they created.
\end{acks}


\end{document}